\documentclass{article}

\PassOptionsToPackage{authoryear}{natbib}
\usepackage[preprint]{neurips_2026}

\usepackage[utf8]{inputenc}
\usepackage[T1]{fontenc}
\usepackage{hyperref}
\usepackage{url}
\usepackage{booktabs}
\usepackage{amsmath,amssymb,amsthm,mathtools}
\usepackage{nicefrac}
\usepackage{microtype}
\usepackage{xcolor}
\usepackage{tikz}
\usepackage{algorithm}
\usepackage{algpseudocode}

\hypersetup{
  colorlinks=true,
  citecolor=blue,
  linkcolor=blue,
  urlcolor=blue
}

\newtheorem{assumption}{Assumption}
\newtheorem{theorem}{Theorem}
\newtheorem{proposition}{Proposition}
\newtheorem{lemma}{Lemma}
\newtheorem{corollary}{Corollary}
\newtheorem{remark}{Remark}

\newcommand{\R}{\mathbb R}
\newcommand{\Pp}{\mathbb P}
\newcommand{\E}{\mathbb E}
\newcommand{\Vol}{\operatorname{Vol}}
\newcommand{\Reach}{\operatorname{Reach}}
\newcommand{\diam}{\operatorname{diam}}
\newcommand{\osc}{\operatorname{osc}}

\newcommand{\TV}{\mathrm{TV}}
\newcommand{\Leb}{\operatorname{Leb}}
\newcommand{\cE}{\mathcal E}
\newcommand{\cU}{\mathcal U}
\newcommand{\wideM}{\widehat M_N}
\newcommand{\whatnu}{\widehat\nu_{N,\kappa}}

\numberwithin{equation}{section} 

\newcommand{\lp}{\left(}
\newcommand{\rp}{\right)}

\newcommand{\Rbb}{\mathbb{R}}
\newcommand{\Sbb}{\mathbb{S}}

\newcommand{\Xbb}{\mathbb{X}}

\newcommand{\Ccal}{\mathcal{C}}

\newcommand{\Ucal}{\mathcal{U}}

\newcommand{\Xcal}{\mathcal{X}}

\newcommand{\Zcal}{\mathcal{Z}}
\usepackage{dutchcal}

\DeclareMathAlphabet{\mathdutchcal}{U}{dutchcal}{m}{n}

\newcommand{\ocal}{\mathbcal{o}}

\newcommand\numberthis{\addtocounter{equation}{1}\tag{\theequation}}

\newtheorem{definition}{Definition}[section]

\newcommand{\beq}{\begin{eqnarray*}}
\newcommand{\eeq}{\end{eqnarray*}}
\newcommand{\beqn}{\begin{eqnarray}}
\newcommand{\eeqn}{\end{eqnarray}}
\newcommand{\ben}{\begin{enumerate}}
\newcommand{\een}{\end{enumerate}}
\newcommand{\bit}{\begin{itemize}}
\newcommand{\eit}{\end{itemize}}
\providecommand{\hide}[1]{}

\newcommand{\eps}{\varepsilon}
\newcommand{\vertiii}[1]{{\left\vert\kern-0.25ex\left\vert\kern-0.25ex\left\vert #1 
    \right\vert\kern-0.25ex\right\vert\kern-0.25ex\right\vert}}

\renewcommand{\epsilon}{\eps}

\newcommand{\x}{\boldsymbol{x}}
\newcommand{\y}{\boldsymbol{y}}

\title{Low Dimensional Sampling under Reconstructed Constraints}
\author{
  Imon Banerjee\\
  Department of Statistics\\
  Purdue University\\
  West Lafayette
  \And
  Riddhiman Bhattacharyya\\
  Department of Statistics\\
  University of California\\ Santa Cruz
}

\begin{document}
\maketitle

\begin{abstract}
We study sampling from a distribution supported on an unknown compact $d$-dimensional $C^2$ manifold $M\subset\mathbb{R}^D$, observed only through i.i.d. uniform points from $M$. We reconstruct the constraint using an adaptive local-convex-hull estimator and target an ambient distribution penalized by squared distance to the reconstruction. Although the reconstructed set may be nonsmooth or fail to be a manifold, its Hausdorff accuracy alone suffices to control the Wasserstein error. We quantify the tradeoff between reconstruction accuracy and penalty strength and show that optimal tuning gives an error proportional to the square root of the Hausdorff error. A planar example proves that this dependence is sharp for the proposed scheme. For a $d$ dimensional constraint reconstructed from $N$ observations, the error becomes $\mathcal{O}( (N/N)^{1/d})$. Finally, we prove uniform geometric ergodicity of a Gaussian random-walk Metropolis--Hastings sampler and combine reconstruction, approximation, and mixing into an explicit finite-time guarantee.
\end{abstract}

\section{Introduction}

Sampling under the presence of a constraint is an important problem in the intersection of machine learning and operations research. When the constraint set is full rank (i.e. has the same dimension as the ambient space), this problem has been well studied by various authors using a variety of techniques. Some examples are: (i) penalization/projection methods that force the sampler to concentrate on the support~\citep{brosse2017proximal,bubeck2018projected}; (ii) using proximal methods ~\citep{chen2022proximal, pmlr-v134-lee21a} ; (iii) constructiong samplers using local entropy~\citep{mohanty2025entropyguidedsamplingflatmodes}. The focus of this paper will be the setting where the constraint set is lower dimensional and unavailable to the practitioner. Instead, we will only assume access to a point cloud satisfying the constraints. We do not assume access to any oracle mechanisms to generate more points in the constraints set. Let $g:\Rbb^D\rightarrow \Rbb$ be a potential. We can write the oracle problem as sampling from
\begin{align*}
 \mu(x)\ & \propto  e^{-g(x)}\\
\text{such that }\quad  & x\in M    \label{eq:oracle_problem} \tag{Oracle Problem}
\end{align*}
Our point-cloud formulation instead constructs a compact set $\widehat M$
from the observations as a proxy for $M$ and penalizes ambient points by
their distance from this set.

Our idea is motivated by \citet[Proposition~4]{brosse2017proximal}, who
consider a full-dimensional compact convex body $A$ and potentials satisfying
their assumptions. Their total-variation bound implies that sampling directly from
\[
\mu_\kappa(x) \propto e^{-g(x)-\kappa d^2(x,A)}  \numberthis\label{eq:brosse_idea}
\]
is $O(\kappa^{-1/2})$ accurate for the constrained target supported on $A$
in the total-variation (TV) metric.


Establishing convergence in the TV metric is traditional default in sampling problems \citep{dalalyan2020sampling, roberts1996geometric, roberts1996exponential, bhattacharya2025explicit, jones2004sufficient,pynadath2024gradient}. For lower-dimensional constraints, the ambient penalized law and the target law on the manifold are mutually singular. Therefore, for our setting, total variation metric is uninformative. We thus state and prove our results in Wasserstein-2 metric which requires us to redo traditional techniques using new tools including the careful reconstruction of the manifold as a non-smooth set from a data cloud, construction of tubular regions within the reach of the manifold where the nearest-point projection is unique.

The heuristic of considering penalisation for lower dimensional constraints can be shown in a concrete example. Let $D=2$
and $M=\mathbb S^1$, the unit circle centered at the origin, and specialize
this illustration to $g\equiv0$. Then $\mu$ is the uniform law on $\mathbb S^1$
with respect to arclength.
Observe that we could write a natural full-dimensional approximation to $\mu$
\[
\mu_\kappa(dx)
\propto
\exp\left\{
-\kappa \frac{(\|x\|-1)^2}{2}
\right\}dx.
\]
$(\|x\|-1)$ is the signed radial displacement of $x$ from $\Sbb^1$
(whose absolute value is subsequently denoted by $d(x,\Sbb^1)$). In polar
coordinates, $\mu_\kappa$ has a uniform angle and radial density proportional
to $\rho e^{-\kappa(\rho-1)^2/2}\mathbf 1_{\{\rho\ge0\}}$. Radial projection
onto $\mathbb S^1$ is optimal because every coupling to a point of
$\mathbb S^1$ has cost at least $(\rho-1)^2$. Gaussian-tail estimates after
the change of variables $t=\sqrt\kappa(\rho-1)$ therefore give
\[
W_2(\mu_\kappa,\mu)={\kappa^{-1/2}}+\ocal(\kappa^{-1})
\]
as \(\kappa\uparrow\infty\).  Our contribution is to lift this heuristic for more general constraint manifolds when we only have a representative point cloud.


The uncertainty in $\widehat M$ is handled using the adaptive local-convex-hull
estimator of \citet{divol2021}, applied to the point cloud
$\Xbb = \{ X_1,\dots,X_N \}$. Its output is a compact reconstructed set and
need not itself be a manifold. Our first formulation is therefore to sample
from the ambient target distribution
\begin{align}
  \mu_\kappa(x) \propto \mathbf 1_K(x)
  e^{-g(x)-(\kappa/2) \Psi_{\wideM}(x)}
  \label{eq:first_formulation}\tag{Manifold Method}
\end{align}
where $\Psi_{\wideM}(x)$ is a $C^2$-capped squared distance between $x$
and $\wideM$. The scalar cap is $C^2$, although the resulting penalty need
not be differentiable when $\wideM$ is rough. The truncation is formally
defined in Section~\ref{sec:model}. 

At this point, we note that the reconstruction not merely a question of statistical accuracy. $\widehat{M}$ may not be smooth, and indeed may not even be a manifold. Consequently our penalty $\Psi$ will be the truncated Hausdorff metric. This introduces novel challenge challenges, which, to the best of our knowledge, has not been considered previously in our generality. We disscuss them in Section \ref{sec:background}.


\textbf{Technical Developments}:
\begin{enumerate}
  \item The principal technical difficulty is that the reconstructed set $\widehat M_N$ need not be smooth, have positive reach, or even be a manifold (see Section~\ref{sec:model}). Consequently, its nearest-point projection may be nonunique and no tubular-coordinate or Jacobian formula is available around $\widehat M_N$. We overcome this difficulty by carrying out the entire comparison in the normal coordinates of the true manifold $M$ (Lemma~\ref{lem:adjusted-potential-expansion}). On the reconstruction event~\eqref{eq:reconstruction-event}, Hausdorff accuracy alone yields upper and lower Gaussian envelopes for the penalty $e^{-\kappa d_{\widehat M_N}^2/2}$ (equation~\eqref{eq:gaussian-distance-envelopes}). Integrating these envelopes along the normal fibers shows that replacing $M$ by $\widehat M_N$ perturbs the projected density by order $\Delta_N\sqrt{\kappa}$ (equations~\eqref{eq:estimated-fiber-uniformity} and~\eqref{eq:projected-marginal}), without requiring any projection or Jacobian associated with the reconstructed set.
  \item Two additional ingredients make this comparison sufficiently sharp. First, we derive a uniform second-order expansion of the Gibbs-weighted normal Jacobian (Lemma~\ref{lem:adjusted-potential-expansion} and equation~\eqref{eq:taylor-expansion-of-adjustment}). The first-order term is odd in the normal displacement and therefore vanishes after integration (see ~\eqref{eq:fiber-uniformity}). Second, we establish a Poincar\'e-to-transport perturbation bound on $M$ (Lemmas~\ref{lem:gibbs-poincare} and~\ref{lem:manifold-density-transport}), which converts the resulting $L^2(\pi_M)$ error in the projected density directly into a Wasserstein--$2$ bound (equation~\eqref{eq:projected-marginal}). Together, these arguments quantify the Wasserstein--$2$ approximation of the true manifold target by the stationary law of the reconstructed-set penalized sampler (Theorem~\ref{thm:main-w2} and equation~\eqref{eq:main-w2-bound}).
\end{enumerate}

\section{Background and Related Literature}\label{sec:background}
Sampling literature is generally vast and has multiple branches. We focus on two particular directions- the Metropolis Hastings (MH) algorithms and fast sampling. For the MH algorithms there have been multiple works which focus on the ergodicity of samplers~\citep{bhattacharya2025explicit, roberts1996geometric, roberts1996exponential,Brown_Jones_2024,vats2018strong}. On the other hand, fast sampling literature has focsed more on mixing times and approximate samplers~\citep{dalalyan2017theoretical, dalalyan2020sampling, durmus2019analysis, durmus2019high, 10.1214/16-AAP1238, mou2019high}. However in all these works, the support of the target is generally the entire space.

There have been works on sampling from constrained spaces~\citep{brosse2017proximal, pmlr-v134-lee21a,chen2022proximal}. However, these works consider the support of the target to be full rank and hence proximal samplers generally work well in these settings.

A separate line of work develops samplers for distributions supported directly on \textit{smooth} lower-dimensional manifolds. Examples include \cite{brubaker2012manifolds,byrne2013geodesic,zappa2018montecarlo,lelievre2019hybrid}. Smoothness of the constraint manifold plays an important role in simplifying their calculations. In contrast, we sample on the reconstructed constraint set, which is non-smooth. Thus techniques developed in prior works do not translate. This requires us to develop new methods of analysis. 

Our reconstruction technique is based on \cite{divol2021}. Complimentary techniques proposed by \cite{fefferman2019} provide smoothened reconstructions, which are less accurate. In particular, our attempts at applying their reconstruction resulted in bounds that are weaker by an order of magnitude. However, sharpening their results is beyond the scope of this paper.




\section{Model and notation}\label{sec:model}

Let $\Rbb^D$ be the ambient space equipped with Euclidean norm and $M$ be a constraint manifold of dimension $d$. We use $q:=D-d$ to denote the codimension. $d_A(x)$ is the distance of $x$ from a nonempty closed set $A\subset\R^D$ and
\[
  \cU_s(A):=\{x\in\R^D:d_A(x)<s\}.
\]
is the tube of radius $s$ around $A$.
The Hausdorff distance between compact sets is denoted by $d_H$.  All
Wasserstein distances below use the ambient Euclidean cost:
\[
 W_2(\mu,\nu)
 :=\left(
   \inf_{\gamma\in\Pi(\mu,\nu)}
   \int_{\R^D\times\R^D}\lVert x-y\rVert^2\,\gamma(dx,dy)
 \right)^{1/2}.
\]
Here $\Pi(\mu,\nu)$ denotes all couplings. For probability measures on
$M$, $W_{2,\mathrm{intrinsic}}$ denotes the analogous Wasserstein distance
formed using the Riemannian geodesic cost.
$\|\cdot\|_{TV}$ denotes total variation under the convention
\[
 \lVert\mu-\nu\rVert_{\TV}
 :=\sup_B|\mu(B)-\nu(B)|
 =\frac12\int|p-s|\,d\zeta,
 \qquad
 \zeta:=\mu+\nu,
 \quad p:=\frac{d\mu}{d\zeta},
 \quad s:=\frac{d\nu}{d\zeta}.
\]
Thus $\lVert\mu-\nu\rVert_{\TV}\in[0,1]$. All probability measures are
understood to be Borel probability measures on their stated spaces.

$\mathcal B(K)$ denotes the Borel $\sigma$-algebra of $K$. When used in
Markov-kernel comparisons, $\mu_1$ and $\mu_2$ denote arbitrary Borel
probability measures on $K$. Unsubscripted $c,C>0$ denote finite constants
that may change from line to line; any restricted dependencies are stated
locally.
$\Leb_D$ denotes $D$-dimensional Lebesgue measure, and $I_D$ denotes the
$D\times D$ identity matrix. For a scalar $s$, write
$s_+:=\max\{s,0\}$. The notation $\delta_x$ denotes the Dirac probability measure at $x$.
All differential-geometric notation and conventions follow
\citet{lee2018riemannian}. In particular, the second fundamental form and
shape operator used later are those of Chapter~8. Reach of a submanifold is defined below.
\begin{definition}
  The reach of a manifold $M$ is the maximum $\tau\in \Rbb^+\cup\{0\}$ such that all points in $\Ucal_\tau(M)$ can be uniquely projected on $M$.
\end{definition}
For $x\in\Ucal_{\Reach(M)}(M)$, let $P_M(x)$ denote the unique nearest
point to $x$ on $M$.

We assume that $M$ is a compact, connected submanifold without boundary
and with positive reach. This regularity condition is popular \citep{aamari2017thesis,aamari2019nonasymptotic,aamari2019reach}, and in our setting, provides uniform
control of the second fundamental form. This
translates directly to a bound on the variation of the
Jacobian in the area formula for $M$.

\begin{assumption}[True manifold and observations]
\label{ass:true-manifold}
For $1\le d<D$, the set $M\subset\R^D$ is a compact, connected,
$d$-dimensional $C^2$ embedded submanifold without boundary, with
\[
  \Reach(M)\ge\tau_0>0.
\]
We assume throughout that $d$ and $\tau_0$ are known. We observe
\[
  X_1,\ldots,X_N\overset{\mathrm{iid}}{\sim}
  \operatorname{Unif}(M)
  :=\frac{\Vol_M}{\Vol_M(M)}.
\]
\end{assumption}
The simplifying distributional assumption on $X_i$ enables cleaner exposition. Our analysis extends to permit a sampling density with respect to $\Vol_M$ that is bounded above and bounded away from zero. Complimentary reconstruction techniques \citep{fefferman2019} allow Gaussian jitters but lose on the accuracy. However, since our main focus is on sampling and not on the accuracy of the reconstruction, we will not focus on these generalisations.

\paragraph{Reconstruction of the constraint set.}
We use all observations $\mathcal X_N:=\{X_1,\ldots,X_N\}$. Retain the
theoretical notation
\begin{equation}
 V_M:=\Vol_M(M).
\label{eq:volume-def}
\end{equation}
The unknown $V_M$ is used only in theoretical constants and is not an input
to the reconstruction.

Following \citet[equation~(1.6) and Definitions~4.1 and~4.4]{divol2021},
for every finite nonempty set $\sigma\subset\R^D$, define
\[
 \operatorname{rad}(\sigma)
 :=\inf_{z\in\R^D}\max_{x\in\sigma}\lVert x-z\rVert.
\]
For a finite nonempty set $S\subset\R^D$ and $t\ge0$, define
\[
 \operatorname{Conv}(t,S)
 :=\bigcup_{\substack{\varnothing\ne\sigma\subseteq S\\
                 \operatorname{rad}(\sigma)\le t}}
       \operatorname{conv}(\sigma),
\]
\[
 h(t,S):=d_H(\operatorname{Conv}(t,S),S),
 \qquad
 \mathcal R(S):=
 \{\operatorname{rad}(\sigma):
       \varnothing\ne\sigma\subseteq S,\ 
       \operatorname{rad}(\sigma)>0\}.
\]
Fix $0<\lambda<3^{-1/d}$ and set
\[
 t_\lambda(S)
 :=\inf\{t\in\mathcal R(S):h(t,S)\le\lambda t\},
 \qquad
 \wideM:=\operatorname{Conv}(t_\lambda(\mathcal X_N),\mathcal X_N).
\]
We use the conventions $\inf\varnothing:=\infty$ and
$\operatorname{Conv}(\infty,S):=\operatorname{conv}(S)$. Excluding zero
from $\mathcal R(S)$ prevents the selector from returning the degenerate
scale zero. These conventions cover degenerate clouds and make $\wideM$ a
nonempty compact set on every outcome. In either case it is a finite union
of compact convex hulls. Since the construction uses only
finitely many subset radii, convex hulls, inequalities, and finite unions,
it is a measurable random compact set (for the Hausdorff Borel structure).
Its inputs are only the cloud and $\lambda$; in particular, $V_M$, a density
bound, a prescribed Hausdorff tolerance, and a smoothness order are not
inputs. Under Assumption~\ref{ass:ambient-envelope} below, the definition
directly gives
\[
 \mathcal X_N\subseteq\wideM
 \subseteq\operatorname{conv}(\mathcal X_N)
 \subseteq\overline B(z_0,R_0).
\]
The output $\wideM$ need not be a manifold, have dimension $d$, have
positive reach, or carry a surface-volume distribution.

For a desired analysis tolerance $\Delta_N>0$, define the event
\begin{equation}
 \cE_N:=\{d_H(M,\wideM)\le\Delta_N\}.
\label{eq:reconstruction-event}\tag{Reconstruction}
\end{equation}
Let $\beta_d:=\max\{d-1,2\}$. The following proposition gives the
probability of this event.
\begin{proposition}[Reconstruction event]
\label{prop:reconstruction-event}
Suppose Assumption~\ref{ass:true-manifold} holds, $1\le d<D$, and the
fixed tuning parameter satisfies $0<\lambda<3^{-1/d}$. There exist
constants $C_H,C_P>0$ and $N_0$ such that, for every $N\ge N_0$,
\begin{equation}
 \Pp\left\{
 d_H(\wideM,M)>
 C_H\left(\frac{\log N}{N}\right)^{2/d}
 \right\}
 \le C_P\frac{(\log N)^{\beta_d}}{N^2}.
\label{eq:divol-reconstruction-tail}
\end{equation}
The constants and threshold may depend on the fixed model, including its
geometry and volume, and on $\lambda$. Consequently, whenever
\eqref{eq:reconstruction-sample-size} holds,
\begin{equation}
  \Pp(\cE_N)\ge1-\alpha_N.
\label{eq:event-probability}
\end{equation}
The probability is over the original observations only.
\end{proposition}

The following assumption is that $M$ sits fully inside a compact ball.
\begin{assumption}[Known ambient enclosure]
\label{ass:ambient-envelope}
There exist known deterministic quantities $z_0\in\R^D$ and $R_0>0$
such that
\[
    M\subseteq \overline B(z_0,R_0).
\]
\end{assumption}
This assumption allows us to choose a fixed compact sampling region that contains the true manifold, the reconstructed set, and the tubes used in our comparison argument. The size of the enclosure enters the finite-time mixing constants and the constants hidden in the approximation bounds, but it does not change the stated convergence-rate exponents.

Fix $0<2r<\tau_0$ and let
\[
 K:=\overline B(z_0,R_0+3r).
\]
This bounded enlargement of the known ambient enclosure is mild.
We define
\[
 L_K:=\diam(K)=2(R_0+3r),
  \qquad
  |K|:=\Leb_D(K).  
\]
We repeatedly use the following elementary consequence of maximal coupling:
for any probability measures $\mu_1$ and $\mu_2$ supported on $K$,
\begin{equation}
 W_2(\mu_1,\mu_2)^2
 \le L_K^2\lVert\mu_1-\mu_2\rVert_{\TV}.
\label{eq:w2-tv-coupling}
\end{equation}
Indeed, a maximal coupling $(Y,Y')$ satisfies
$\Pp(Y\ne Y')=\lVert\mu_1-\mu_2\rVert_{\TV}$, while
$\lVert Y-Y'\rVert^2\le L_K^2\mathbf1_{\{Y\ne Y'\}}$.

To control the penalty outside the tubular neighborhood, we apply a $C^2$
truncation to the squared distance. Let
$\chi:[0,\infty)\to[0,\infty)$ be nondecreasing and $C^2$, with
\begin{align}
  \chi(s)&=s &&\text{for }0\le s\le r^2,\label{eq:cap-inner}\\
  \chi(s)&\ge r^2 &&\text{for }s\ge r^2,\label{eq:cap-lower}\\
  \chi(s)&=\chi(4r^2) &&\text{for }s\ge4r^2,\label{eq:cap-outer}\\
  C_\chi&:=\sup_{s\ge0}|\chi'(s)|<\infty.
\label{eq:cap-lipschitz}
\end{align}
For every nonempty compact $A\subset\R^D$, define
\begin{equation}
  \Psi_A(x):=\chi(d_A(x)^2),\qquad x\in\R^D.
\label{eq:capped-penalty}
\end{equation}
Observe that distance is $1$-Lipschitz and the derivative of
$s\mapsto\chi(s^2)$ is bounded by $4rC_\chi$. For every set $A$,
\[
 0\le\Psi_A\le\chi(4r^2),
 \qquad
 |\Psi_A(x)-\Psi_A(x')|
 \le4rC_\chi\lVert x-x'\rVert.
\]
The penalty agrees with
$d_A^2$ on $\cU_r(A)$ and is at least $r^2$ outside that neighborhood.
For $A=M$, the nearest-point projection is $C^1$ and $d_M^2$ is $C^2$ on
$\cU_{2r}(M)$; the constant cap therefore makes $\Psi_M$ globally $C^2$.

\begin{assumption}[Potential and target distribution]
\label{ass:potential-target}
The potential $g$ is $C^2$ on a neighborhood of $K$. Define the finite
constants and assume they are bounded
\begin{align}
 G&:=\sup_{x\in K}\lVert\nabla g(x)\rVert,\label{eq:G-def}\\
 H&:=\sup_{x\in K}\lVert\nabla^2g(x)\rVert_{\mathrm{op}},\label{eq:H-def}\\
 \osc_K(g)&:=\sup_Kg-\inf_Kg.\label{eq:osc-def}
\end{align}
The desired distribution is the singular Gibbs law
\begin{equation}
  \pi_M(dy)
  :=\frac{e^{-g(y)}}{Z_M}\,\Vol_M(dy),
  \qquad
  Z_M:=\int_Me^{-g(y)}\,\Vol_M(dy).
\label{eq:true-target}
\end{equation}
\end{assumption}
The analysis below is deterministic on $\cE_N$ and uses only its Hausdorff
accuracy.

\section{Algorithm and Main Results}

For a penalty strength
$\kappa>0$, define the true-manifold and
reconstructed-set penalized laws on $K$ by
\begin{align}
 \nu_{M,\kappa}(dx)
 &:=\frac{\mathbf 1_K(x)}{Z_{M,\kappa}^{K}}
   \exp\left\{-g(x)-\frac\kappa2\Psi_M(x)\right\}dx,
\label{eq:true-penalized}\\
 \whatnu(dx)
 &:=\frac{\mathbf 1_K(x)}{\widehat Z_{N,\kappa}^{K}}
   \exp\left\{-g(x)-\frac\kappa2\Psi_{\wideM}(x)\right\}dx.
\label{eq:estimated-penalized}
\end{align}
Here $Z_{M,\kappa}^{K}$ and $\widehat Z_{N,\kappa}^{K}$ are the respective
normalizing constants.
The factor $1/2$ is a normalization convention: in the normal direction,
the leading profile is a Gaussian with variance $\kappa^{-1}$.
The penalty strength $\kappa$ and proposal scale $\sigma$ are held fixed
throughout each run of the sampler.

\begin{algorithm}[!ht]
\caption{Gaussian RMWH Using the Reconstructed Constraint Set}
\label{alg:rwmh}
\begin{algorithmic}[1]
\Require Initial state $X^{(0)}\in K$, proposal scale $\sigma>0$,
penalty strength $\kappa>0$, and number of iterations $M_{\mathrm{iter}}$
\For{$m=0,\ldots,M_{\mathrm{iter}}-1$}
  \State Draw $Z^{(m)}\sim\mathcal N(0,I_D)$ and set
  $Y^{(m)}\gets X^{(m)}+\sigma Z^{(m)}$
  \If{$Y^{(m)}\notin K$}
    \State $X^{(m+1)}\gets X^{(m)}$
  \Else
    \State Draw $U^{(m)}\sim\operatorname{Unif}(0,1)$
    \If{$U^{(m)}\le
    a_{N,\kappa}(X^{(m)},Y^{(m)})$}
      \State $X^{(m+1)}\gets Y^{(m)}$
    \Else
      \State $X^{(m+1)}\gets X^{(m)}$
    \EndIf
  \EndIf
\EndFor
\State \Return $X^{(1)},\ldots,X^{(M_{\mathrm{iter}})}$
\end{algorithmic}
\end{algorithm}

\paragraph{Random-walk Metropolis--Hastings.}
We approximate $\whatnu$ using a Gaussian random-walk
Metropolis--Hastings chain. Fix a proposal scale $\sigma>0$ and initialize
$X^{(0)}\in K$. Given $X^{(m)}=x$, draw independently
\[
 Z^{(m)}\sim\mathcal N(0,I_D),
 \qquad
 Y^{(m)}:=x+\sigma Z^{(m)}.
\]
Since the proposal density is symmetric, the Metropolis--Hastings
acceptance probability, for $x\in K$ and $y\in\R^D$, is
\begin{equation}
 a_{N,\kappa}(x,y)
 :=
 \begin{cases}
 \displaystyle
 \min\left\{
  1,
  \exp\left(
    g(x)-g(y)
    -\dfrac\kappa2
     \bigl[\Psi_{\wideM}(y)-\Psi_{\wideM}(x)\bigr]
  \right)
 \right\}, & y\in K,\\[2ex]
 0, & y\notin K.
 \end{cases}
\label{eq:rwmh-acceptance}
\end{equation}
Draw $U^{(m)}\sim\operatorname{Unif}(0,1)$ independently and set
\begin{equation}
 X^{(m+1)}
 :=
 \begin{cases}
  Y^{(m)}, & U^{(m)}\le a_{N,\kappa}(X^{(m)},Y^{(m)}),\\
  X^{(m)}, & \text{otherwise}.
 \end{cases}
\label{eq:rwmh-update}
\end{equation}
Thus proposals outside $K$ are rejected before $g(Y^{(m)})$ or the
exponential acceptance ratio is evaluated. Writing
\[
 q_\sigma(z)
 :=(2\pi\sigma^2)^{-D/2}
   \exp\left(-\frac{\lVert z\rVert^2}{2\sigma^2}\right),
\]
the transition kernel is
\begin{align}
 P_{N,\kappa}(x,dy)
 :={}&q_\sigma(y-x)a_{N,\kappa}(x,y)\,dy
 +r_{N,\kappa}(x)\delta_x(dy),
\label{eq:rwmh-kernel}\\
 r_{N,\kappa}(x)
 :={}&1-
 \int_{\R^D}q_\sigma(y-x)a_{N,\kappa}(x,y)\,dy.
\notag
\end{align}

The symmetry of $q_\sigma$ implies detailed balance with respect to $\whatnu$; hence $\whatnu$ is invariant for $P_{N,\kappa}$. In an implementation, the values of $g$ and $\Psi_{\wideM}$ at the current state can be retained from the preceding iteration, so each step requires evaluating them only at the proposal.

For the quantitative guarantee, define
\begin{align}
 C_d
 &:=4(d+1)\left(\frac32\right)^d,\\
 A_M
 &:=C_d e^{Gr+\frac12Hr^2 }
 \left(
 H+G^2+\frac{G}{\tau_0}+\frac1{\tau_0^2}
 \right),
\label{eq:AM-def}\\
 \kappa_0&:=\max\left\{1,\frac{8q}{r^2},\,2A_Mq\right\}.
\end{align}
For $\kappa,\sigma>0$, also set
\begin{equation}
 \underline\varepsilon_{\kappa,\sigma}
 :=|K|(2\pi\sigma^2)^{-D/2}
   \exp\left\{
     -\frac{L_K^2}{2\sigma^2}
     -\osc_K(g)
     -\frac{\kappa}{2}\chi(4r^2)
   \right\}.
\label{eq:rwmh-deterministic-minorization}
\end{equation}

To state our main theorem we require some additional notation. Let $C_H,C_P>0$ and $N_0$ be the
fixed-model constants from Proposition~\ref{prop:reconstruction-event}.
For $\Delta_N>0$ and $\alpha_N\in(0,1)$, the reconstruction sample-size
condition is
\begin{equation}
 N\ge N_0,
 \qquad
 C_H\left(\frac{\log N}{N}\right)^{2/d}\le\Delta_N,
 \qquad
 C_P\frac{(\log N)^{\beta_d}}{N^2}\le\alpha_N.
\label{eq:reconstruction-sample-size}
\end{equation}

The main result of this paper is the following theorem.

\begin{theorem}[Finite-time RWMH approximation of the manifold target]
\label{thm:rwmh-end-to-end}
Suppose Assumptions~\ref{ass:true-manifold},
\ref{ass:ambient-envelope}, and~\ref{ass:potential-target} hold, and let
$K$ and $r$ be chosen as in Assumption~\ref{ass:ambient-envelope}. Fix $\kappa\ge\kappa_0$ and $\sigma>0$,
and suppose
\[
 0<\Delta_N\le r/4,
 \qquad
 \Delta_N\sqrt\kappa\le1.
\]
Then $\whatnu$ is the unique invariant probability measure of
$P_{N,\kappa}$. There is a fixed-model constant $C_M>0$ such that,
for every initial law $\rho$ on $K$ and every integer $m\ge0$,
on the reconstruction event $\cE_N$,
\begin{align}
 W_2(\rho P_{N,\kappa}^m,\pi_M)
 \le{}&L_K(1-\underline\varepsilon_{\kappa,\sigma})^{m/2}
 \notag\\
 &+C_M\left(
 \kappa^{-1/2}+\Delta_N\sqrt\kappa
 +\kappa^{q/4}e^{-\kappa r^2/64}
 \right).
\label{eq:rwmh-end-to-end}
\end{align}
If the reconstruction sample-size condition
\eqref{eq:reconstruction-sample-size} holds, then this bound holds with
probability at least $1-\alpha_N$.
\end{theorem}
The constant $C_M$ appearing here and in the stationary approximation bounds
below depends only on the fixed model, including the finite weighted
Poincar\'e constant of $\pi_M$. It is independent of $N$, $\kappa$, and the
Hausdorff tolerance, and is not an input to the reconstruction or the sampler.
The proof of Theorem~\ref{thm:rwmh-end-to-end} involves three major steps:
\begin{enumerate}
 \item \textbf{Approximation of the stationary law.}
 We compare the reconstructed penalized law $\whatnu$ with the manifold
 target $\pi_M$. The core result is Theorem~\ref{thm:main-w2}, enabled by the
 Taylorisation of the Gibbs-weighted Jacobian
 (Lemma~\ref{lem:adjusted-potential-expansion}), an integrated
 Gaussian-envelope estimate, and transport on the true manifold
 (Lemma~\ref{lem:manifold-density-transport}). These results appear in Section~\ref{sec:stationary-approximation}.

 \item \textbf{Finite-time convergence of the sampler.}
 We compare the law $\rho P_{N,\kappa}^m$ of the RWMH chain with its
 invariant law $\whatnu$.  Proposition~\ref{prop:rwmh-geometric} gives the
 required geometric convergence bound and contributes the first term in
 \eqref{eq:rwmh-end-to-end}. It appears in
 Section~\ref{sec:rwmh-geometric-convergence}.

 \item \textbf{High-probability reconstruction.}
 We show that the geometric reconstruction event $\cE_N$, on which the
 stationary-law comparison is valid, holds with probability at least
 $1-\alpha_N$.  This is Proposition~\ref{prop:reconstruction-event}, proved
 in Section~\ref{sec:reconstruction-proof}. This step produces the stated high-probability guarantee.
\end{enumerate}

\eqref{eq:rwmh-end-to-end} depicts a need to choose a correct penalty $\kappa$.  so called "Goldilock-zone" of choosing the correct rate. The following corollary provides a rate of convergence of the expected Wasserstein loss. Its proof appears in Section \ref{sec:proofs}.
\begin{corollary}[Complexity With Respect to $N$]
\label{cor:rwmh-expected-rate}
Under the hypotheses of Theorem~\ref{thm:rwmh-end-to-end}, let
$0<\epsilon_{\mathrm{tar}}<\min\{1,L_K\}$ be a target Wasserstein accuracy. We use the estimator above and make the analysis choices
\[
 \Delta_N:=\epsilon_{\mathrm{tar}}^2,
 \qquad
 \kappa_N:=\epsilon_{\mathrm{tar}}^{-2},
 \qquad
 \alpha_N:=\epsilon_{\mathrm{tar}}.
\]
Fix $\sigma>0$, and suppose $\epsilon_{\mathrm{tar}}$ is small enough that
$\kappa_N\ge\kappa_0$ and
$\epsilon_{\mathrm{tar}}^2\le r/4$. Suppose $N$ is large enough that the
reconstruction sample-size condition \eqref{eq:reconstruction-sample-size}
holds,
and choose an integer $m_N\ge0$ such that
\begin{equation}
 m_N
 \ge
 \left\lceil
 \frac{2}{\underline\varepsilon_{\kappa_N,\sigma}}
 \log\frac{L_K}{\epsilon_{\mathrm{tar}}}
\right\rceil,
\label{eq:rwmh-optimized-iterations}
\end{equation}
Then, for every initial law $\rho$ on $K$,
\begin{equation}
 \E\left[
 W_2(\rho P_{N,\kappa_N}^{m_N},\pi_M)
 \right]
 =O(\epsilon_{\mathrm{tar}}).
\label{eq:rwmh-expected-rate}
\end{equation}
The implicit constant is independent of $N$ and
$\epsilon_{\mathrm{tar}}$. In particular, suppressing log-factors
\[
 N=\widetilde O\left(
 \epsilon_{\mathrm{tar}}^{-d}
 \right),
\]
is a sufficient observation count for fixed model parameters.
\end{corollary}


\subsection{Approximation of the Manifold Target by the Stationary Law}
\label{sec:stationary-approximation}

The main result of this section is the following comparison between the
stationary law of the penalized sampler and the desired manifold target.
\begin{theorem}
\label{thm:main-w2}
Suppose Assumptions~\ref{ass:true-manifold},
\ref{ass:ambient-envelope}, and~\ref{ass:potential-target} hold, and $K$
and $r$ are chosen as above. Suppose
\[
 \kappa\ge\kappa_0,
 \qquad 0<\Delta_N\le r/4,
 \qquad \Delta_N\sqrt\kappa\le1.
\]
There is a fixed-model constant $C_M>0$, independent of
$N,\kappa,\Delta_N$, such that, on $\cE_N$,
\begin{equation}
 W_2(\whatnu,\pi_M)
 \le C_M\left(
 \kappa^{-1/2}+\Delta_N\sqrt\kappa
 +\kappa^{q/4}e^{-\kappa r^2/64}
 \right).
\label{eq:main-w2-bound}
\end{equation}
Consequently, if \eqref{eq:reconstruction-sample-size} holds, this bound
holds with probability at least $1-\alpha_N$ over the observations.
\end{theorem}

We now analyze the reconstructed penalized law directly in the normal
coordinates of $M$. Only these true-manifold coordinates are used; no
projection or Jacobian associated with $\wideM$ is required.

For $y\in M$, let
$\mathrm{I\,I}_y:T_yM\times T_yM\to N_yM$ denote the second fundamental
form. For $v\in N_yM$, define the shape operator
$S_{y,v}:T_yM\to T_yM$ by
\[
 \langle S_{y,v}u,w\rangle
 :=\langle \mathrm{I\,I}_y(u,w),v\rangle,
 \qquad u,w\in T_yM.
\]
For $y\in M$ and $v\in N_yM$ with $\lVert v\rVert<r$, define
\[
 J_M(y,v):=\det(I-S_{y,v}),
 \qquad
 \ell_y(v):=-\langle\nabla g(y),v\rangle
 -\operatorname{tr}(S_{y,v}),
\]
and
\[
 R_y(v):=e^{-[g(y+v)-g(y)]}J_M(y,v)-1-\ell_y(v).
\]
Our main result in this section is enabled by the following lemma. 
\begin{lemma}[Taylorisation of the Gibbs-weighted Jacobian]
\label{lem:adjusted-potential-expansion}
Suppose Assumptions~\ref{ass:true-manifold},
\ref{ass:ambient-envelope}, and~\ref{ass:potential-target} hold, and $K$ and
$r$ are chosen as above. Every
$x\in\Ucal_r(M)$ has unique normal coordinates $x=y+v$, where $y=P_M(x)$
and $v\in N_yM$ with $\lVert v\rVert<r$, and
\[
 dx=J_M(y,v)\,dv\,\Vol_M(dy).
\]
Moreover, $\ell_y$ is linear, and hence odd, in $v$, and
\[
 e^{-[g(y+v)-g(y)]}J_M(y,v)
 =1+\ell_y(v)+R_y(v),
 \qquad
 |R_y(v)|\le A_M\lVert v\rVert^2.\numberthis\label{eq:taylor-expansion-of-adjustment}
\]
\end{lemma}
The following lemma provides a Poincare inequality for expectations on the manifold. The proof is a combination of spectral bounds of a manifold and standard perturabations. This lemma might be known, but we could not find a reference anywhere. So we provide a proof in Section \ref{sec:proofs}.
\begin{lemma}[Poincar\'e inequality for the Gibbs law]
\label{lem:gibbs-poincare}
Under Assumptions~\ref{ass:true-manifold} and
\ref{ass:potential-target}, there exists a constant
$C_{\mathrm{PI}}<\infty$ such that, for every smooth function
$\varphi:M\to\R$,
\[
 \int_M\left|\varphi-\int_M\varphi\,d\pi_M\right|^2d\pi_M
 \le C_{\mathrm{PI}}\int_M\|\nabla_M\varphi\|^2d\pi_M.
\]
Thus one may take
\[
 C_{\mathrm{PI}}
 \le \frac{\exp\{\sup_M g-\inf_M g\}}{\lambda_1(M)}.
\]
\end{lemma}

\begin{lemma}[Transport of a density perturbation on $M$]
\label{lem:manifold-density-transport}
Under Assumptions~\ref{ass:true-manifold} and
\ref{ass:potential-target}, let $C_{\mathrm{PI}}<\infty$ be any constant
for which the inequality in Lemma~\ref{lem:gibbs-poincare} holds.
If $\mu=h\pi_M$ is a probability measure,
$h\in L^2(\pi_M)$, and $h\ge b>0$ almost everywhere, then
\[
 W_2(\mu,\pi_M)
 \le
 \sqrt{\frac{C_{\mathrm{PI}}}{\min\{1,b\}}}\,
 \|h-1\|_{L^2(\pi_M)}.
\]
Here $W_2$ uses the ambient Euclidean cost.
\end{lemma}

The next proposition compares a Hausdorff-accurate reconstructed-set
penalty directly with the manifold target.
\begin{proposition}
\label{prop:true-penalty}
Suppose Assumptions~\ref{ass:true-manifold},
\ref{ass:ambient-envelope}, and~\ref{ass:potential-target} hold, and $K$
and $r$ are chosen as above. Let $A\subset K$ be a nonempty compact set
with $d_H(A,M)\le\Delta$, where
\[
 0\le\Delta\le r/4,
 \qquad \kappa\ge\kappa_0,
 \qquad \Delta\sqrt\kappa\le1.
\]
Define
\[
 \nu_{A,\kappa}(dx)
 :=\frac{\mathbf1_K(x)}{Z_{A,\kappa}^K}
 e^{-g(x)-\kappa\Psi_A(x)/2}\,dx.
\]
Here $Z_{A,\kappa}^K$ is the normalizing constant.
There is a fixed-model constant $C_M>0$, independent of $A,\Delta,\kappa$,
such that
\begin{equation}
 W_2(\nu_{A,\kappa},\pi_M)
 \le C_M\left(
 \kappa^{-1/2}+\Delta\sqrt\kappa
 +\kappa^{q/4}e^{-\kappa r^2/64}
 \right).
\label{eq:true-penalty-bound}
\end{equation}
\end{proposition}

The following geometric observation supplies the fixed ambient enclosure
used in the main theorem.
\begin{lemma}[Fixed enclosure of the relevant distance neighborhoods]
\label{lem:common-tubes-on-event}
Suppose Assumptions~\ref{ass:true-manifold}
and~\ref{ass:ambient-envelope} hold. For every realization of the noiseless
cloud and its reconstructed set,
\[
  2r<\Reach(M),
\]
and
\begin{equation}
  \overline{\cU_{2r}(M)}
  \cup
  \overline{\cU_{2r}(\wideM)}
  \subset\operatorname{int}(K).
\label{eq:automatic-complete-tubes}
\end{equation}
In particular, $M\cup\wideM\subset\operatorname{int}(K)$.
\end{lemma}

The terms in \eqref{eq:main-w2-bound} represent, respectively,
normal displacement, the distortion of the projected density caused
by reconstructing $M$, and exponentially small mass outside the true
normal tube. The true-manifold fiber bias is of order $\kappa^{-1}$
and has been absorbed into the normal-displacement term.

\begin{corollary}[Small-error rate]
\label{cor:optimized}
Under the assumptions of Theorem~\ref{thm:main-w2}, choose
$\kappa=\Delta_N^{-1}$ and suppose
$0<\Delta_N\le\min\{1,r/4,\kappa_0^{-1}\}$.
Then, on $\cE_N$,
\begin{equation}
 W_2(\whatnu,\pi_M)=O(\Delta_N^{1/2}).
\label{eq:Delta-half}
\end{equation}
\end{corollary}

\begin{remark}[Rate inherited from reconstruction]\label{remark:rate}
Take
\[
 \Delta_N=C_H\left(\frac{\log N}{N}\right)^{2/d},
 \qquad
 \kappa_N\asymp\left(\frac{N}{\log N}\right)^{2/d}.
\]
Then Proposition~\ref{prop:reconstruction-event} and the calculation in
the proof of Corollary~\ref{cor:optimized}, applied to
\eqref{eq:main-w2-bound}, give
\[
 W_2(\widehat\nu_{N,\kappa_N},\pi_M)
 =O_{\Pp}\left(\left(\frac{\log N}{N}\right)^{1/d}\right).
\]
The same rate holds in expectation. 
\end{remark}

\subsection{Geometric convergence of the random-walk Metropolis-Hastings}
\label{sec:rwmh-geometric-convergence}

We now control the error caused by running the sampler for finitely many
iterations.

\begin{proposition}[Uniform ergodicity of Gaussian RWMH]
\label{prop:rwmh-geometric}
Let $K$ be the closed ball of positive radius used in
\eqref{eq:estimated-penalized}, and suppose
Assumption~\ref{ass:potential-target} holds. Fix $\kappa>0$, $\sigma>0$,
and a nonempty compact realization of $\wideM$, and let
$\underline\varepsilon_{\kappa,\sigma}$ be defined by
\eqref{eq:rwmh-deterministic-minorization}.
Then $0<\underline\varepsilon_{\kappa,\sigma}<1$, and the kernel
$P_{N,\kappa}$ defined in \eqref{eq:rwmh-kernel} satisfies
\begin{equation}
 P_{N,\kappa}(x,A)
 \ge
 \underline\varepsilon_{\kappa,\sigma}
 \frac{\Leb_D(A)}{|K|},
 \qquad x\in K,\quad A\in\mathcal B(K).
 \label{eq:rwmh-uniform-minorization}
\end{equation}
Moreover, $P_{N,\kappa}$ is reversible with respect to $\whatnu$,
and $\whatnu$ is its unique invariant probability measure. For every
initial law $\rho$ on $K$ and every integer $m\ge0$,
\begin{align}
 \lVert\rho P_{N,\kappa}^{m}-\whatnu\rVert_{\TV}
 &\le(1-\underline\varepsilon_{\kappa,\sigma})^m
      \lVert\rho-\whatnu\rVert_{\TV}
 \le(1-\underline\varepsilon_{\kappa,\sigma})^m,
 \label{eq:rwmh-tv-rate}\\
 W_2(\rho P_{N,\kappa}^{m},\whatnu)
 &\le L_K(1-\underline\varepsilon_{\kappa,\sigma})^{m/2}.
 \label{eq:rwmh-w2-rate}
\end{align}
Thus the chain is uniformly geometrically ergodic. 
\end{proposition}

\begin{remark}
  The displayed constant does not depend on the realized reconstruction.
\end{remark}

The proposition holds for every fixed nonempty compact reconstruction independent of its geometric properties. The reconstruction event is needed for the separate comparison of $\whatnu$ with $\pi_M$.

For any desired sampler error $0<\eta<L_K$, the inequality
$1-u\le e^{-u}$ gives the sufficient iteration count
\[
 m\ge\left\lceil
 \frac{2}{\underline\varepsilon_{\kappa,\sigma}}
 \log\frac{L_K}{\eta}
 \right\rceil
 \quad\Longrightarrow\quad
 W_2(\rho P_{N,\kappa}^m,\whatnu)\le\eta.
\]
This establishes geometric convergence for each fixed $\kappa$ and
$\sigma$. The displayed rate is conservative: the reciprocal of
$\underline\varepsilon_{\kappa,\sigma}$ grows exponentially in
$\kappa$ when the other quantities are fixed. No convergence rate
uniform in $\kappa$ or in the ambient dimension is asserted.

\section{Sharpness of the Hausdorff-Based Stationary Rate}

Corollary~\ref{cor:optimized} obtains an error of order
$\Delta_N^{1/2}$ optimising over $\kappa$. The following theorem shows that it cannot be improved by further optimisation. Thus, the performance of the sampler is bottlenecked by the performance of the reconstruction in a tight way.

\begin{theorem}[Sharpness uniformly over the penalty strength]
\label{thm:uniform-penalty-lower-bound}
Consider the fixed compact convex body
\[
 \Omega:=\left\{x\in\R^2:
 \bigl((|x_1|-2)_+\bigr)^4
 +\bigl((|x_2|-1)_+\bigr)^4\le1\right\}.
\]
Its boundary $M:=\partial\Omega$ is a $C^3$ curve satisfying
Assumptions~\ref{ass:true-manifold} and~\ref{ass:ambient-envelope}. It
contains straight edges and curved corners, with curvature bounded below by
a positive constant on a closed portion of a curved corner. Fix $K$, $r$,
and $\chi$ as in
Section~\ref{sec:model}.
For every fixed $\eta\in(0,1/3)$, there are constants $c,C,t_0>0$ such
that, whenever $0<t<t_0$ and a finite observation set
$\mathcal X_N\subset M$ satisfies
\begin{equation}
 d_H(\mathcal X_N,M)\le\eta t,
\label{eq:lower-bound-observation-density}
\end{equation}
the reconstruction
\[
 A_{t,N}:=\operatorname{Conv}(t,\mathcal X_N),
 \qquad
 \Delta:=d_H(A_{t,N},M),
\]
satisfies
\[
 ct^2\le\Delta\le Ct^2 \numberthis\label{eq:lower-bound-basic-geometry-main}
\]
and
\[
 \inf_{\kappa>0}W_2(\nu_{A_{t,N},\kappa},\pi_M)
 \ge ct
 \ge c\sqrt\Delta.
\]
The constants are independent of $t$, $\mathcal X_N$, and $\kappa$.
\end{theorem}

The exponent $1/2$ in the Hausdorff-based stationary approximation
bound is therefore sharp for our family of samplers for every choice of $\kappa$. Like is usual with lower bounds, the proof of this Theorem is quite technical, and is therefore relegated to Section \ref{sec:prf-lowerbound}.


\section{Example: A Gibbs law on a torus}
\label{sec:examples}

\begingroup
\renewcommand{\theHfigure}{torus.\arabic{figure}}
\refstepcounter{figure}
\label{fig:torus-reach}
\begin{center}
 \begin{tikzpicture}[
   x=0.9cm,
   y=0.9cm,
   line cap=round,
   line join=round,
   every node/.style={font=\small}
 ]
   \path[fill=blue!24,draw=blue!58!black,line width=0.9pt,even odd rule]
     (0,0) ellipse (3.65 and 1.55)
     (0,0) ellipse (1.55 and 0.58);
   \draw[blue!58!black,line width=0.8pt] (0,0) ellipse (3.65 and 1.55);
   \draw[blue!58!black,line width=0.8pt] (0,0) ellipse (1.55 and 0.58);
   \draw[blue!42,densely dotted] (0,0.43) ellipse (2.55 and 0.82);
   \draw[blue!42,densely dotted] (0,-0.43) ellipse (2.55 and 0.82);
   \draw[black!55,dashed,line width=0.8pt] (0,0) ellipse (2.55 and 0.86);

   \path[fill=orange!48,draw=orange!75!black,line width=1.0pt]
     (0,0) ellipse (1.55 and 0.58);

   \fill (0,0) circle (2.1pt);
   \draw[black,line width=1.0pt] (0,0)--(1.55,0);
   \node[above] at (0.78,0) {$1$};
   \node[above left=-1pt] at (0,0) {$c$};
   \node[blue!58!black] at (-3.05,1.72) {$M$};
   \node[orange!75!black,align=center] at (0,-0.95) { };
 \end{tikzpicture}
 \par\smallskip
 \begin{minipage}{0.94\linewidth}
 \small
 \textbf{Figure~\thefigure:}
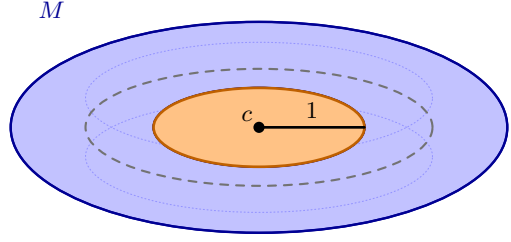
 A schematic of the ring torus used in the example, with major radius $3$
 and minor radius $1$. The orange circle is a flat disk of radius $1$ in
 the equatorial plane, drawn to cover the displayed central opening. It diagrametically
 depicts the reach $\tau_0=1$. We mathematically verify it below.
 \end{minipage}
\end{center}
\endgroup

Consider the ring torus with major radius $3$ and minor radius $1$,
parametrized by
\begin{equation}
 F(\theta,\phi)
 :=\bigl((3+\cos\theta)\cos\phi,
         (3+\cos\theta)\sin\phi,
         \sin\theta\bigr),
 \qquad (\theta,\phi)\in[0,2\pi)^2,
\label{eq:torus-parametrization}
\end{equation}
and let $M:=F([0,2\pi)^2)$. Take the ambient potential
\begin{equation}
 g(x):=x_3,
 \qquad x=(x_1,x_2,x_3)\in\R^3.
\label{eq:torus-potential}
\end{equation}
The target in \eqref{eq:true-target} is therefore
\[
 \pi_M(dy)=Z_M^{-1}e^{-y_3}\,\Vol_M(dy).
\]
It favors the lower part of the torus, where $y_3<0$, while remaining
strictly positive on all of $M$.

For a completely explicit cap, let
\[
 w(t):=1-3t^2+2t^3,
 \qquad 0\le t\le1,
\]
and define
\[
 \chi(s):=
 \begin{cases}
  s, &0\le s\le r^2,\\[1mm]
  \displaystyle
  r^2+3r^2\int_0^{(s-r^2)/(3r^2)}w(t)\,dt,
       &r^2<s<4r^2,\\[3mm]
  \displaystyle\frac52r^2, &s\ge4r^2.
 \end{cases}
\]

\begin{proposition}[Torus specialization of Corollary~\ref{cor:rwmh-expected-rate}]
\label{prop:torus-corollary}
For the torus $M$, potential $g$, and cap $\chi$ defined above, we have 
for every initial law $\rho$ on $K$,
\[
 \E\left[W_2\left(
 \rho P_{N,\epsilon_{\mathrm{tar}}^{-2}}^{m_N},\pi_M
 \right)\right]
 =O(\epsilon_{\mathrm{tar}}),
\]
with the sample-size scaling
$N=\widetilde O(\epsilon_{\mathrm{tar}}^{-2})$;
the implicit constant being independent of $N$ and
$\epsilon_{\mathrm{tar}}$.We also have $\sigma>0$ and a target accuracy satisfies
\begin{equation}
 0<\epsilon_{\mathrm{tar}}
 \le(162e^{1/4})^{-1/2}\approx0.0693,
\label{eq:torus-target-range}
\end{equation}
with
\[
 \Delta_N=\epsilon_{\mathrm{tar}}^2,
 \qquad
 \kappa_N=\epsilon_{\mathrm{tar}}^{-2},
 \qquad
 \alpha_N=\epsilon_{\mathrm{tar}}.
\]
Suppose also that the reconstruction sample-size condition
\eqref{eq:reconstruction-sample-size} holds. Since $d=2$, this is simply
\eqref{eq:reconstruction-sample-size} with $\beta_d=2$.
Finally,
\begin{equation}
 \begin{aligned}
 m_N\ge\Bigg\lceil
 &\frac{96}{6859\pi}(2\pi\sigma^2)^{3/2}
 \exp\left\{
  \frac{361}{8\sigma^2}+\frac{19}{2}
  +\frac{5}{64\epsilon_{\mathrm{tar}}^2}
 \right\}\\
 &\hspace{28mm}\times
 \log\left(\frac{19}{2\epsilon_{\mathrm{tar}}}\right)
 \Bigg\rceil.
 \end{aligned}
\label{eq:torus-iteration-count}.
\end{equation}

\end{proposition}

\section{Conclusion}

The present draft focuses on establishing the validity of sampling from
reconstructed constraints. A fuller treatment develops complementary
conductance-based mixing-time bounds for the RWMH chain under additional
conditions. We defer their precise statements, constants, and proofs to the
full submission, since they are not needed for the core approximation
guarantee sought here.

Beyond this, two directions remain open. First, we seek central limit
theorems and consistent Monte Carlo error estimates for parameter inference
on reconstructed constraints, adapting general MCMC limit theory and
batch-means methods~\citep{roberts1996geometric,flegal2010batch}. Second, it
would be useful to determine the optimal scaling of the Gaussian proposal
variance with the penalty strength, the ambient and intrinsic dimensions,
and the reconstruction error. Such a result would quantify the
acceptance--exploration trade-off as the target concentrates near the
reconstructed set and connect our setting to classical diffusion-limit
analyses of random-walk Metropolis algorithms~\citep{roberts1997weak,roberts2001optimal}.

\appendix

\section{Appendix}

\label{sec:proofs}

\subsection{Proof of Proposition~\ref{prop:reconstruction-event}}

\label{sec:reconstruction-proof}

\begin{proof}
Throughout this proof, $C>0$ may change from line to line and depend only
on the fixed model and on $\lambda$.
The sampling density with respect to $\Vol_M$ equals $1/V_M$, so it has
positive finite lower and upper bounds. Thus the fixed
law belongs to the model used by \citet{divol2021}, with reach lower bound
$\tau_0$. Let $t^*(\mathcal X_N)$ be as defined in their equation~(3.4).
$M$ and $t^*(\mathcal X_N)$ are not used by the
reconstruction algorithm. They are required only for the purpose of the proof.

By \citet[Proposition~3.9, with parameter $a=3$]{divol2021}, for all
sufficiently large $N$,
\[
 \Pp\left\{
 d_H(\mathcal X_N,M)>C\left(\frac{\log N}{N}\right)^{1/d}
 \right\}
 \le C\frac{(\log N)^{d-1}}{N^2}.
\]
By \citet[Theorem~4.6(1), with $b=2$]{divol2021}, the restriction
$\lambda<3^{-1/d}$ implies that, outside an event of probability at most
$C(\log N)^{\beta_d}/N^2$,
\[
 t^*(\mathcal X_N)
 \le t_\lambda(\mathcal X_N)
 \le\frac{t^*(\mathcal X_N)}{\lambda}
 \left[
  1+C\left(\frac{(\log N)^2}{N}\right)^{1/d}
 \right].
\]
On the intersection of these two events, take $N$ large enough that
$d_H(\mathcal X_N,M)<\tau_0/8$ and the last bracket is at most $2$.
\citet[Proposition~3.4]{divol2021} then gives
\[
 \begin{aligned}
 d_H(\mathcal X_N,M)
 &\le C\left(\frac{\log N}{N}\right)^{1/d},\\
 t^*(\mathcal X_N)
 &\le2d_H(\mathcal X_N,M),\\
 t^*(\mathcal X_N)
 &\le t_\lambda(\mathcal X_N)
 \le\frac{2}{\lambda}t^*(\mathcal X_N).
 \end{aligned}
\]
In particular, the selected scale is finite and tends to zero, so all the
small-scale requirements in the cited results hold after increasing the
fixed threshold $N_0$.

For $t_\lambda(\mathcal X_N)>t^*(\mathcal X_N)$,
\citet[Lemma~3.3]{divol2021} directly yields
\[
 d_H(\wideM,M)
 \le\frac{t_\lambda(\mathcal X_N)^2}{\Reach(M)}
 \le\frac{t_\lambda(\mathcal X_N)^2}{\tau_0}.
\]
The same conclusion holds at the possible endpoint
$t_\lambda(\mathcal X_N)=t^*(\mathcal X_N)$. Indeed, for a finite cloud and
a finite selected scale, $\operatorname{Conv}(t,\mathcal X_N)$ is unchanged
for $t$ just to the right of the selected critical radius. Apply
Lemma~3.3 at such scales and let them decrease to
$t_\lambda(\mathcal X_N)$. Therefore,
\[
 d_H(\wideM,M)
 \le\frac{t_\lambda(\mathcal X_N)^2}{\tau_0}
 \le C_H\left(\frac{\log N}{N}\right)^{2/d}.
\]
A union bound, using $d-1\le\beta_d$, proves
\eqref{eq:divol-reconstruction-tail}. The two achieved-error inequalities
in \eqref{eq:reconstruction-sample-size} then imply
\eqref{eq:event-probability}.
\end{proof}

\subsection{Proof of Theorem~\ref{thm:rwmh-end-to-end}}

\begin{proof}
Proposition~\ref{prop:rwmh-geometric} gives uniqueness of the invariant
law and, for every $\rho$ and $m$,
\[
 W_2(\rho P_{N,\kappa}^m,\whatnu)
 \le L_K(1-\underline\varepsilon_{\kappa,\sigma})^{m/2}.
\]
On $\cE_N$, the triangle inequality and Theorem~\ref{thm:main-w2} give
\[
 W_2(\rho P_{N,\kappa}^m,\pi_M)
 \le W_2(\rho P_{N,\kappa}^m,\whatnu)
    +W_2(\whatnu,\pi_M),
\]
and substitution of the two established bounds yields
\eqref{eq:rwmh-end-to-end}. The probability statement follows from
Proposition~\ref{prop:reconstruction-event}.
\end{proof}

\subsection{Proof of Corollary~\ref{cor:rwmh-expected-rate}}

\begin{proof}
The inequality $1-u\le e^{-u}$ and
\eqref{eq:rwmh-optimized-iterations} give
\[
 L_K(1-\underline\varepsilon_{\kappa_N,\sigma})^{m_N/2}
 \le\epsilon_{\mathrm{tar}}.
\]
The analysis tolerance is $\Delta_N=\epsilon_{\mathrm{tar}}^2$, and hence
$\kappa_N=\epsilon_{\mathrm{tar}}^{-2}=\Delta_N^{-1}$.
Corollary~\ref{cor:optimized} therefore gives
\[
 W_2(\whatnu,\pi_M)=O(\epsilon_{\mathrm{tar}})
 \qquad\text{on }\cE_N.
\]
The triangle inequality therefore yields
\[
 W_2(\rho P_{N,\kappa_N}^{m_N},\pi_M)
 =O(\epsilon_{\mathrm{tar}})
 \qquad\text{on }\cE_N.
\]
For every reconstruction outcome, both measures in this Wasserstein
distance are supported on $K$, so the distance is at most $L_K$.
For $N$ as in the hypothesis of Corollary~\ref{cor:rwmh-expected-rate},
Proposition~\ref{prop:reconstruction-event} gives
$\Pp(\cE_N^c)\le\epsilon_{\mathrm{tar}}$, and consequently
\[
 \E\left[
 W_2(\rho P_{N,\kappa_N}^{m_N},\pi_M)
 \right]
 \le C\epsilon_{\mathrm{tar}}
   +L_K\Pp(\cE_N^c)
 \le(C+L_K)\epsilon_{\mathrm{tar}}
 =O(\epsilon_{\mathrm{tar}}).
\]
Here $L_K=2(R_0+3r)$ is fixed. For fixed model parameters and sufficiently
small $\epsilon_{\mathrm{tar}}$, a sufficiently large multiple of
$\epsilon_{\mathrm{tar}}^{-d}\log(1/\epsilon_{\mathrm{tar}})$ satisfies
the sample-size hypothesis of Corollary~\ref{cor:rwmh-expected-rate}; the
failure-tail condition is then of smaller order. This proves the stated
sufficient observation count $N=\widetilde O(\epsilon_{\mathrm{tar}}^{-d})$.
\end{proof}

\subsection{Proof of Theorem~\ref{thm:main-w2}}

\begin{proof}
Lemma~\ref{lem:common-tubes-on-event} places the true normal tube inside
$K$ on every realization. On $\cE_N$, apply
Proposition~\ref{prop:true-penalty} with $A=\wideM$ and
$\Delta=\Delta_N$. The law $\nu_{A,\kappa}$ is then exactly $\whatnu$,
so \eqref{eq:true-penalty-bound} gives \eqref{eq:main-w2-bound}.
Finally, Proposition~\ref{prop:reconstruction-event} and
\eqref{eq:reconstruction-sample-size} give
$\Pp(\cE_N)\ge1-\alpha_N$ whenever the stated sample-size condition
holds.
\end{proof}

\subsection{Proof of Lemma~\ref{lem:adjusted-potential-expansion}}

\begin{proof}
Since $K$ and $r$ are chosen as above, $r<\tau_0\le\Reach(M)$ and
$\Ucal_r(M)\subset K$. Thus every $x\in\Ucal_r(M)$ has the unique representation
\[
  x=y+v,
  \qquad y=P_M(x)\in M,
  \qquad v\in N_yM,
  \qquad \lVert v\rVert<r.
\]
For $y\in M$, let
$\mathrm{I\,I}_y:T_yM\times T_yM\to N_yM$ denote the second fundamental
form. For $v\in N_yM$, define the shape operator
$S_{y,v}:T_yM\to T_yM$ by
\[
 \langle S_{y,v}u,w\rangle
 :=\langle \mathrm{I\,I}_y(u,w),v\rangle,
 \qquad u,w\in T_yM.
\]
The operator $S_{y,v}$ is self-adjoint and depends linearly on $v$ \citep[eq. (8.4)]{lee2018riemannian}. By the endpoint-map differential formula \citep[Proposition~4.1.8]{palais1988} and change of variables, the Euclidean volume element in these
coordinates is
\begin{equation}
 dx=J_M(y,v)\,dv\,\Vol_M(dy),
 \qquad
 J_M(y,v)=\det(I-S_{y,v}),
\label{eq:tube-jacobian}
\end{equation}
By \citet[Proposition~A.1(i)]{aamari2019reach}, every unit vector
$u\in T_yM$ satisfies
\[
 \lVert\mathrm{I\,I}_y(u,u)\rVert
 \le \frac{1}{\Reach(M)}
 \le \frac{1}{\tau_0}.
\]
Therefore,
\[
 \lVert S_{y,v}\rVert_{\mathrm{op}}
 =\sup_{\lVert u\rVert=1}
   |\langle S_{y,v}u,u\rangle|
 \le \frac{\lVert v\rVert}{\tau_0}.
\]
In particular, $\lVert S_{y,v}\rVert_{\mathrm{op}}<1/2$ because
$\lVert v\rVert<r<\tau_0/2$. 

Let $\lambda_1(y,v),\ldots,\lambda_d(y,v)$ be the eigenvalues of
$S_{y,v}$. Since $S_{y,v}$ is self-adjoint,
\begin{align*}
|\lambda_i(y,v)|\le\lVert v\rVert/\tau_0<1/2\numberthis\label{eq:eigenvalue-bound}  
\end{align*}
for every $i$. Thus
\begin{align*}
 J_M(y,v) & = \prod_{i=1}^d \lp 1-\lambda_i(y,v)\rp\\
 & =1-\operatorname{tr}(S_{y,v})
  +\underbrace{\sum_{k=2}^d(-1)^k
    \sum_{1\le i_1<\cdots<i_k\le d}
    \prod_{j=1}^k\lambda_{i_j}(y,v)}_{\text{remainder}}
\label{eq:jacobian-expansion}\numberthis
\end{align*}
where we have used
$\operatorname{tr}(S_{y,v})=\sum_{i=1}^d\lambda_i(y,v)$.

We now analyse the remainder. Set
$a:=\lVert v\rVert/\tau_0<1/2$. For each $k$, it contains
$\binom{d}{k}$ products of $k$ eigenvalues; by
\eqref{eq:eigenvalue-bound} each product has absolute value at most $a^k$.
Combining this fact with the triangle inequality gives
\begin{align*}
 \bigg|\text{remainder}\bigg|
 &\le \sum_{k=2}^d\binom{d}{k}a^k\\
 &\le4\left[\left(\frac32\right)^d-1-\frac d2\right]a^2
 \le C_da^2
 =C_d\frac{\lVert v\rVert^2}{\tau_0^2}.
\end{align*}
To expand the potential term, define
\[
 h_{y,v}(t):=e^{-[g(y+tv)-g(y)]},
 \qquad 0\le t\le1.
\]
Then $h_{y,v}(0)=1$, and differentiation gives
\begin{align*}
 h_{y,v}'(t)
 &=-\langle\nabla g(y+tv),v\rangle h_{y,v}(t),\\
 h_{y,v}''(t)
 &=\left(
    \langle\nabla g(y+tv),v\rangle^2
    -\langle v,\nabla^2g(y+tv)v\rangle
   \right)h_{y,v}(t).
\end{align*}
Because the segment $\{y+tv:0\le t\le1\}$ lies in $\Ucal_r(M)\subset K$, the definitions of $G$ \eqref{eq:G-def} and mean-value theorem
\begin{align*}
 |g(y+tv)-g(y)|&\le Gr.
\end{align*}
Using this, along with definition of $H$ \eqref{eq:H-def} we get
\begin{align*}
 |h_{y,v}''(t)|
 &\le e^{Gr}(G^2+H)\lVert v\rVert^2.
\end{align*}
Taylor's theorem with integral remainder therefore yields
\begin{equation}
 e^{-[g(y+v)-g(y)]}
 =1-\langle\nabla g(y),v\rangle
 +\underbrace{\int_0^1(1-t)h_{y,v}''(t)\,dt}_{\text{remainder 2}},
\label{eq:potential-expansion}
\end{equation}
and the remainder satisfies
\[
 \left|\int_0^1(1-t)h_{y,v}''(t)\,dt\right|
 \le \frac12e^{Gr}(G^2+H)\lVert v\rVert^2
\].
We now collect the terms that are linear in $v$.  Set
\[
 \ell_y(v)
 :=-\langle\nabla g(y),v\rangle
   -\operatorname{tr}(S_{y,v}),
\]
and define the nonlinear remainder by
\begin{equation}
 R_y(v)
 :=e^{-[g(y+v)-g(y)]}J_M(y,v)-1-\ell_y(v).
\label{eq:Ry-def}
\end{equation}
Equivalently,
\begin{equation}
 e^{-[g(y+v)-g(y)]}J_M(y,v)
 =1+\ell_y(v)+R_y(v).
\label{eq:fiber-expansion}
\end{equation}
To bound $R_y(v)$, expand its definition as
\begin{align*}
 R_y(v)
 ={}&\left[J_M(y,v)-1+\operatorname{tr}(S_{y,v})\right]\\
 &-\langle\nabla g(y),v\rangle\left[J_M(y,v)-1\right]\\
 &+\left[e^{-[g(y+v)-g(y)]}-1
          +\langle\nabla g(y),v\rangle\right]J_M(y,v).
\end{align*}
The first bracket is the remainder in
\eqref{eq:jacobian-expansion}, and the third bracket is remainder~2 in
\eqref{eq:potential-expansion}.  Moreover,
\eqref{eq:eigenvalue-bound} implies
\[
|J_M(y,v)|\le\left(\frac32\right)^d,
\qquad
|J_M(y,v)-1|
 \le d\left(\frac32\right)^{d-1}a
 \le C_d\frac{\lVert v\rVert}{\tau_0}.
\]
The same fixed $C_d=4(d+1)(3/2)^d$ also dominates
$\frac12(3/2)^d$, the coefficient arising from the potential remainder.
Combining these three bounds gives
\[
 |R_y(v)|
 \le C_d\left(
 H+G^2+\frac{G}{\tau_0}+\frac{1}{\tau_0^2}
 \right)e^{Gr+Hr^2/2}\lVert v\rVert^2
 =A_M\lVert v\rVert^2,
\]
where the last equality uses the definition of $A_M$ in
\eqref{eq:AM-def}.
\end{proof}

\subsection{Proof of Lemma~\ref{lem:gibbs-poincare}}

\begin{proof}
Let $\sigma_M:=\Vol_M/V_M$ be normalized Riemannian volume on $M$.
The spectral characterization of the Poincar\'e inequality gives
\begin{equation}
 \int_M\left|\varphi-\int_M\varphi\,d\sigma_M\right|^2d\sigma_M
 \le \frac{1}{\lambda_1(M)}
 \int_M\|\nabla_M\varphi\|^2d\sigma_M.
\label{eq:volume-poincare}
\end{equation}
By \cite{cheeger1970lower} (see also Corollary, \cite{yau1975isoperimetric}), $\lambda_1(M)$ (which is the first non-zero eigenvalue of the Laplace--Beltrami operator of the compact connected Riemannian manifold $M$) is positive.


The rest of the proof is the bounded-perturbation argument of
\citet{holley1987logarithmic}. Since
\[
 \frac{d\pi_M}{d\sigma_M}(y)=\frac{V_Me^{-g(y)}}{Z_M},
 \qquad
 d\sigma_M(y)=\frac{Z_Me^{g(y)}}{V_M}\,d\pi_M(y),
\]
we have
\begin{align*}
 \int_M\left|\varphi-\int_M\varphi\,d\pi_M\right|^2d\pi_M
 &=\inf_{c\in\R}\int_M|\varphi-c|^2d\pi_M\\
 &\le \frac{V_Me^{-\inf_Mg}}{Z_M}
       \int_M\left|\varphi-\int_M\varphi\,d\sigma_M\right|^2
       d\sigma_M\\
 &\le \frac{V_Me^{-\inf_Mg}}{Z_M\lambda_1(M)}
       \int_M\|\nabla_M\varphi\|^2d\sigma_M\\
 &\le \frac{\exp\{\sup_Mg-\inf_Mg\}}{\lambda_1(M)}
       \int_M\|\nabla_M\varphi\|^2d\pi_M.
\end{align*}
Thus one may take
\[
 C_{\mathrm{PI}}
 =\frac{\exp\{\sup_Mg-\inf_Mg\}}{\lambda_1(M)},
\]
which proves the claim.
\end{proof}

\subsection{Proof of Lemma~\ref{lem:manifold-density-transport}}

\begin{proof}
Lemma~\ref{lem:gibbs-poincare} and the Lax--Milgram theorem
\citep[Corollary~5.8]{brezis2011functional} give a mean-zero weak solution
$u$ of
\[
 \int_M\langle\nabla_Mu,\nabla_M\varphi\rangle\,d\pi_M
 =\int_M(h-1)\varphi\,d\pi_M.
\]
Taking $\varphi=u$ and applying Cauchy--Schwarz and the Poincar\'e
inequality yields
\[
 \int_M\|\nabla_Mu\|^2d\pi_M
 \le C_{\mathrm{PI}}\|h-1\|_{L^2(\pi_M)}^2.
\]
For $0\le t\le1$, put
$h_t=1+t(h-1)$ and
$v_t=\nabla_Mu/h_t$.
Since $h_t\ge\min\{1,b\}$, this velocity has finite action.
For every smooth test function $\varphi$,
\[
 \frac{d}{dt}\int_M\varphi h_t\,d\pi_M
 =\int_M\langle\nabla_M\varphi,v_t\rangle h_t\,d\pi_M,
\]
so $(h_t\pi_M,v_t)$ satisfies the continuity equation.
The Benamou--Brenier formula for the intrinsic Riemannian distance gives
\[
 W_{2,\mathrm{intrinsic}}(\mu,\pi_M)^2
 \le\int_0^1\int_M\frac{\|\nabla_Mu\|^2}{h_t}\,d\pi_M\,dt
 \le
 \frac{C_{\mathrm{PI}}}{\min\{1,b\}}
 \|h-1\|_{L^2(\pi_M)}^2.
\]
This proves the result.
\end{proof}

\subsection{Proof of Proposition~\ref{prop:true-penalty}}

\begin{proof}
Throughout this proof, $c,C>0$ may change between displays and depend
only on the fixed model, not on $A,\Delta,\kappa$.

\paragraph{Part I. Inside the tube:}
Let $\nu_{A,\kappa}^{\mathrm{in}}$ be $\nu_{A,\kappa}$ conditioned on
$\cU_{r/2}(M)$, and let
\[
 \mu_\kappa:=(P_M)_\#\nu_{A,\kappa}^{\mathrm{in}}.
\]
This true tube lies in $K$ and has unique normal coordinates.
Also, $d_A(x)\le d_M(x)+\Delta<3r/4$ there, so the cap is inactive.
For $X\sim\nu_{A,\kappa}^{\mathrm{in}}$, the coupling $(X,P_M(X))$ gives
\begin{equation}
 W_2(\nu_{A,\kappa}^{\mathrm{in}},\mu_\kappa)
 \le
 \left(\E_{\nu_{A,\kappa}^{\mathrm{in}}}d_M(X)^2\right)^{1/2}
 \le C\kappa^{-1/2},
\label{eq:projection-coupling}
\end{equation}
where the last inequality is proved below.

Let

\begin{align}
 c_{\kappa,q,r/2}
 &:=\int_{\lVert v\rVert<r/2}
       e^{-\kappa\lVert v\rVert^2/2}\,dv,\label{eq:c-kappa}\\
 a_\kappa(y)
 &:=\int_{\lVert v\rVert<r/2}
 e^{-[g(y+v)-g(y)]}J_M(y,v)
 e^{-\kappa\lVert v\rVert^2/2}\,dv.
\label{eq:fiber-weight}
\end{align}
For the reconstructed set, define
\begin{equation}
 a_{A,\kappa}(y)
 :=\int_{\lVert v\rVert<r/2}
 e^{-[g(y+v)-g(y)]}J_M(y,v)
 e^{-\kappa d_A(y+v)^2/2}\,dv.
\label{eq:estimated-fiber-weight}
\end{equation}
We first expand the reference weight $a_\kappa(y)$ using
Lemma~\ref{lem:adjusted-potential-expansion}.

{\color{black}
Substituting \eqref{eq:taylor-expansion-of-adjustment} into \eqref{eq:fiber-weight} gives
}
\begin{align*}
 a_\kappa(y)
 ={}&c_{\kappa,q,r/2}
 +\int_{\lVert v\rVert<r/2}\ell_y(v)
 e^{-\kappa\lVert v\rVert^2/2}\,dv\\
 &+\int_{\lVert v\rVert<r/2}R_y(v)
 e^{-\kappa\lVert v\rVert^2/2}\,dv.
\end{align*}
The ball $\{v:\lVert v\rVert<r/2\}$ is symmetric, the Gaussian factor is
even, and $\ell_y$ is odd.  Therefore
\[
 \int_{\lVert v\rVert<r/2}\ell_y(v)  e^{-\kappa\lVert v\rVert^2/2}\,dv =0.
\]
Consequently,
\[
 a_\kappa(y)-c_{\kappa,q,r/2}
 =\int_{\lVert v\rVert<r/2}R_y(v)
 e^{-\kappa\lVert v\rVert^2/2}\,dv.
\]
{\color{black}
By Lemma~\ref{lem:adjusted-potential-expansion},
$|R_y(v)|\le A_M\lVert v\rVert^2$. Therefore,
}
\[
 |a_\kappa(y)-c_{\kappa,q,r/2}|
 \le A_M\int_{\lVert v\rVert<r/2}\lVert v\rVert^2
 e^{-\kappa\lVert v\rVert^2/2}\,dv.
\]
Let $Z\sim\mathcal N(0,\kappa^{-1}I_q)$. Dividing by
$c_{\kappa,q,r/2}$ and observing that conditioning on $\lVert Z\rVert<r/2$
can only decrease the second moment gives
\begin{align*}
 \left|\frac{a_\kappa(y)}{c_{\kappa,q,r/2}}-1\right|
 &\le A_M\,
 \frac{\int_{\lVert v\rVert<r/2}\lVert v\rVert^2
 e^{-\kappa\lVert v\rVert^2/2}\,dv}
 {\int_{\lVert v\rVert<r/2}e^{-\kappa\lVert v\rVert^2/2}\,dv}\\
 &=A_M\E\!\left[\lVert Z\rVert^2
 \mid \lVert Z\rVert<r/2\right]
 \le\frac{A_Mq}{\kappa}.
\end{align*}
Thus
\begin{equation}
 \sup_{y\in M}
 \left|\frac{a_\kappa(y)}{c_{\kappa,q,r/2}}-1\right|
 \le\frac{A_Mq}{\kappa}.
\label{eq:fiber-uniformity}
\end{equation}
For $\kappa\ge\kappa_0$, $A_Mq/\kappa\le1/2$.
For $\|v\|<r/2$, Hausdorff accuracy gives
\[
 |d_A(y+v)-\|v\||\le\Delta,
\]
and therefore
\begin{equation}
 e^{-\kappa(\|v\|+\Delta)^2/2}
 \le e^{-\kappa d_A(y+v)^2/2}
 \le e^{-\kappa(\|v\|-\Delta)_+^2/2}.
\label{eq:gaussian-distance-envelopes}
\end{equation}
We use the following elementary radial estimate, where $C_q<\infty$
depends only on $q$:
\[
 \int_{\mathbb R^q}
 \left[
 e^{-(\|z\|-a)_+^2/2}-e^{-(\|z\|+a)^2/2}
 \right]dz
 \le C_q a,\qquad 0\le a\le1.
\]
Indeed, the integral is zero at $a=0$. In polar coordinates its
derivative with respect to $a$ is bounded uniformly on $[0,1]$ by a
constant multiple of
$\int_0^\infty u(u+1)^{q-1}e^{-u^2/2}\,du<\infty$.
The mean-value theorem proves the estimate.
Applying it with $z=\sqrt\kappa\,v$ and $a=\sqrt\kappa\,\Delta$
gives an integrated envelope discrepancy at most
$C\Delta\kappa^{-(q-1)/2}$.
Only these radial envelopes are integrated over the whole normal
space; the actual-distance integrals stay in $\|v\|<r/2$.

The factor $e^{-[g(y+v)-g(y)]}J_M(y,v)$ is bounded above and bounded
away from zero uniformly for $y\in M$, $\|v\|<r/2$.
This follows from $|g(y+v)-g(y)|\le G\|v\|$ and
$(1-\|v\|/\tau_0)^d\le J_M(y,v)\le(1+\|v\|/\tau_0)^d$.
Consequently,
\[
 |a_{A,\kappa}(y)-a_\kappa(y)|
 \le C\Delta\kappa^{-(q-1)/2}.
\]
Since $\kappa r^2\ge8q$,
Markov's inequality for $Z\sim N(0,\kappa^{-1}I_q)$ gives
$\Pr(\|Z\|\ge r/2)\le1/2$. Hence
$c_{\kappa,q,r/2}\asymp\kappa^{-q/2}$.
Combining this with \eqref{eq:fiber-uniformity} yields
\begin{equation}
 \sup_{y\in M}
 \left|\frac{a_{A,\kappa}(y)}{c_{\kappa,q,r/2}}-1\right|
 \le C\left(\kappa^{-1}+\Delta\sqrt\kappa\right).
\label{eq:estimated-fiber-uniformity}
\end{equation}

We also need a uniform positive lower bound, even when the right-hand
side of this last display is not small.
For $u\ge0$, the inequalities
$(u+\Delta)^2\le2u^2+2\Delta^2$ and
$(u-\Delta)_+^2\ge u^2/2-\Delta^2$, together with
$\kappa\Delta^2\le1$, give on the tube
\[
 e^{-1}e^{-\kappa\|v\|^2}
 \le e^{-\kappa d_A(y+v)^2/2}
 \le e^{1/2}e^{-\kappa\|v\|^2/4}.
\]
Integrating these bounds with the uniformly positive, bounded
Gibbs-weighted Jacobian factor gives
\begin{equation}
 c\le\frac{a_{A,\kappa}(y)}{c_{\kappa,q,r/2}}\le C,
 \qquad y\in M.
\label{eq:estimated-fiber-positive}
\end{equation}
For the lower bound, the truncated Gaussian with exponent
$-\kappa\|v\|^2$ also has mass comparable to $\kappa^{-q/2}$.

Every $x\in\Ucal_{r/2}(M)$ has the unique
representation
\[
  x=y+v,
  \qquad y=P_M(x)\in M,
  \qquad v\in N_yM,
  \qquad \lVert v\rVert<r/2.
\]
{\color{black}
The corresponding area formula is
$dx=J_M(y,v)\,dv\,\Vol_M(dy)$.

By the formula for conditional expectation, we obtain, for every measurable $B\subset M$,
}
\[
 \mu_\kappa(B)
 =\frac{
 \displaystyle\int_B\int_{\lVert v\rVert<r/2}
 e^{-g(y+v)}e^{-\kappa d_A(y+v)^2/2}J_M(y,v)
 \,dv\,\Vol_M(dy)}{
 \displaystyle\int_M\int_{\lVert v\rVert<r/2}
 e^{-g(y+v)}e^{-\kappa d_A(y+v)^2/2}J_M(y,v)
 \,dv\,\Vol_M(dy)}.
\]
{\color{black}
For each fixed $y\in M$, factor
\[
 e^{-g(y+v)}=e^{-g(y)}e^{-[g(y+v)-g(y)]}.
\]
}
Therefore, by \eqref{eq:estimated-fiber-weight}, the inner integral satisfies
\begin{align*}
 &\int_{\lVert v\rVert<r/2}
 e^{-g(y+v)}e^{-\kappa d_A(y+v)^2/2}J_M(y,v)\,dv\\
 &\qquad=e^{-g(y)}\int_{\lVert v\rVert<r/2}
 e^{-[g(y+v)-g(y)]}e^{-\kappa d_A(y+v)^2/2}
 J_M(y,v)\,dv\\
 &\qquad=e^{-g(y)}a_{A,\kappa}(y).
\end{align*}
{\color{black}
Substituting this identity into both the numerator and the denominator gives
}
\[
 \mu_\kappa(B)
 =\frac{\displaystyle\int_B e^{-g(y)}a_{A,\kappa}(y)\,\Vol_M(dy)}
 {\displaystyle\int_M e^{-g(y)}a_{A,\kappa}(y)\,\Vol_M(dy)}.
\]
{\color{black}
By comparison,
\[
 \pi_M(B)
 =\frac{\displaystyle\int_B e^{-g(y)}\,\Vol_M(dy)}
 {\displaystyle\int_M e^{-g(y)}\,\Vol_M(dy)}.
\]
}
Hence, writing
\[
 b_\kappa(y):=\frac{a_{A,\kappa}(y)}{c_{\kappa,q,r/2}},
 \qquad
 \overline b_\kappa:=\int_M b_\kappa(y)\,\pi_M(dy),
\]
{\color{black}
we have
\[
 \frac{d\mu_\kappa}{d\pi_M}(y)
 =\frac{b_\kappa(y)}{\overline b_\kappa}.
\]
}
By \eqref{eq:estimated-fiber-positive},
$c\le b_\kappa(y),\overline b_\kappa\le C$.
In particular, $d\mu_\kappa/d\pi_M$ is bounded away from zero uniformly
in $A,\Delta,\kappa$. Moreover,
\[
 \left|
 \frac{b_\kappa(y)}{\overline b_\kappa}-1
 \right|
 \le
 \frac{|b_\kappa(y)-1|+|\overline b_\kappa-1|}
 {\overline b_\kappa}
 \le C\left(\kappa^{-1}+\Delta\sqrt\kappa\right),
\]
where the last inequality follows from
\eqref{eq:estimated-fiber-uniformity}.
Apply Lemma~\ref{lem:manifold-density-transport} with
$h=b_\kappa/\overline b_\kappa$.
Its $L^2(\pi_M)$ error is no larger than the displayed uniform error,
so
\begin{equation}
 W_2(\mu_\kappa,\pi_M)
 \le C\left(\kappa^{-1}+\Delta\sqrt\kappa\right).
\label{eq:projected-marginal}
\end{equation}

We next control the normal displacement. The same normal-coordinate
formula gives
\begin{align}
 &\E_{\nu_{A,\kappa}^{\mathrm{in}}}d_M(X)^2\notag\\
 &\quad=
 \frac{\displaystyle
 \int_M e^{-g(y)}
 \int_{\|v\|<r/2}\|v\|^2
 e^{-[g(y+v)-g(y)]}J_M(y,v)
 e^{-\kappa d_A(y+v)^2/2}\,dv\,\Vol_M(dy)}
 {\displaystyle
 \int_M e^{-g(y)}a_{A,\kappa}(y)\,\Vol_M(dy)}.
\label{eq:normal-moment-ratio}
\end{align}
By \eqref{eq:estimated-fiber-positive}, the denominator is at least
$cZ_M\kappa^{-q/2}$.
The upper Gaussian envelope already established above bounds the
numerator by
\[
 CZ_M\int_{\mathbb R^q}\|v\|^2e^{-\kappa\|v\|^2/4}\,dv
 =
 CZ_M\frac{2q}{\kappa}
 \left(\frac{4\pi}{\kappa}\right)^{q/2}.
\]
Thus
\begin{equation}
 \E_{\nu_{A,\kappa}^{\mathrm{in}}}d_M(X)^2
 \le C\kappa^{-1}.
\label{eq:normal-thickness-bound}
\end{equation}
This proves \eqref{eq:projection-coupling}.
Together with \eqref{eq:projected-marginal}, the triangle inequality
and $\kappa\ge1$ give
\begin{equation}
 W_2(\nu_{A,\kappa}^{\mathrm{in}},\pi_M)
 \le C\left(\kappa^{-1/2}+\Delta\sqrt\kappa\right).
\label{eq:inside-bound}
\end{equation}

\paragraph{Part II. Outside the tube:}
If $x\in K\setminus\cU_{r/2}(M)$, then
\[
 d_A(x)\ge d_M(x)-\Delta\ge r/4.
\]
Monotonicity of $\chi$ and its agreement with the identity on
$[0,r^2]$ imply $\Psi_A(x)\ge r^2/16$ there. Consequently,
\[
 \int_{K\setminus\cU_{r/2}(M)}
 e^{-g(x)-\kappa\Psi_A(x)/2}\,dx
 \le |K|e^{-\inf_Kg}e^{-\kappa r^2/32}.
\]
On the other hand, \eqref{eq:estimated-fiber-positive} gives
\begin{equation}
 Z_{A,\kappa}^K
 \ge
 \int_M e^{-g(y)}a_{A,\kappa}(y)\,\Vol_M(dy)
 \ge cZ_M\kappa^{-q/2}.
\label{eq:penalized-normalizer-lower}
\end{equation}
Dividing yields
\begin{equation}
 \nu_{A,\kappa}(K\setminus\cU_{r/2}(M))
 \le C\kappa^{q/2}e^{-\kappa r^2/32}.
\label{eq:outside-mass}
\end{equation}
The total variation distance between a law and its conditioning on
an event is the probability of the complementary event. Therefore,
the existing maximal-coupling inequality \eqref{eq:w2-tv-coupling}
gives
\begin{equation}
 W_2(\nu_{A,\kappa},\nu_{A,\kappa}^{\mathrm{in}})
 \le L_K
 \sqrt{\nu_{A,\kappa}(K\setminus\cU_{r/2}(M))}
 \le C\kappa^{q/4}e^{-\kappa r^2/64}.
\label{eq:outside-w2}
\end{equation}
Combining this with \eqref{eq:inside-bound}, and taking $C_M$ large
enough to dominate the fixed constants, proves
\eqref{eq:true-penalty-bound}.
\end{proof}

\subsection{Proof of Lemma~\ref{lem:common-tubes-on-event}}

\begin{proof}
Assumption~\ref{ass:true-manifold} and the choice of $r$ give
$2r<\tau_0\le\Reach(M)$. Moreover, the noiseless observations and the
definition of the estimator give
\[
 M\subseteq\overline B(z_0,R_0),
 \qquad
 \wideM\subseteq\operatorname{conv}(\mathcal X_N)
 \subseteq\overline B(z_0,R_0)
\]
on every realization.

Let $A$ denote either $M$ or $\wideM$. If
$x\in\overline{\cU_{2r}(A)}$, then $d_A(x)\le2r$. Since
$A\subseteq\overline B(z_0,R_0)$, the definition of distance gives
\[
 \lVert x-z_0\rVert
 \le R_0+d_A(x)
 \le R_0+2r<R_0+3r.
\]
Thus $\overline{\cU_{2r}(A)}\subset\operatorname{int}(K)$ for both choices
of $A$, proving \eqref{eq:automatic-complete-tubes}. In particular,
$M\cup\wideM\subset\operatorname{int}(K)$.
\end{proof}

\subsection{Proof of Corollary~\ref{cor:optimized}}

\begin{proof}
These choices imply $\kappa\ge\kappa_0$ and
$\Delta_N\sqrt\kappa=\sqrt{\Delta_N}\le1$.
Substituting into \eqref{eq:main-w2-bound} gives
\[
 W_2(\whatnu,\pi_M)
 \le C_M\left(
 2\Delta_N^{1/2}
 +\Delta_N^{-q/4}e^{-r^2/(64\Delta_N)}
 \right).
\]
The final term is $o(\Delta_N^{1/2})$, proving
\eqref{eq:Delta-half}.
\end{proof}

\subsection{Proof of Proposition~\ref{prop:rwmh-geometric}}

\begin{proof}
The distance to a nonempty compact set is continuous. Thus
$\Psi_{\wideM}=\chi(d_{\wideM}^2)$ is continuous, even without a
reach bound. Since $g$ is continuous on the compact set $K$,
\begin{align*}
 0
 &<
 |K|\exp\left\{
 -\sup_{z\in K}
 \left[g(z)+\frac{\kappa}{2}\Psi_{\wideM}(z)\right]
 \right\}
 \le \widehat Z_{N,\kappa}^{K},\\
 \widehat Z_{N,\kappa}^{K}
 &\le
 |K|\exp\left\{
 -\inf_{z\in K}
 \left[g(z)+\frac{\kappa}{2}\Psi_{\wideM}(z)\right]
 \right\}
 <\infty.
\end{align*}
Consequently, $\whatnu$ in \eqref{eq:estimated-penalized} is a
well-defined probability measure with finite first and second moments.

Since $q_\sigma(y-x)=q_\sigma(x-y)$,
\begin{align*}
 &\frac{
 e^{-g(x)-\kappa\Psi_{\wideM}(x)/2}}
 {\widehat Z_{N,\kappa}^{K}}
 q_\sigma(y-x)a_{N,\kappa}(x,y)\\
 &\qquad=
 \frac{q_\sigma(y-x)}{\widehat Z_{N,\kappa}^{K}}
 \min\left\{
 e^{-g(x)-\kappa\Psi_{\wideM}(x)/2},
 e^{-g(y)-\kappa\Psi_{\wideM}(y)/2}
 \right\}\\
 &\qquad=
 \frac{
 e^{-g(y)-\kappa\Psi_{\wideM}(y)/2}}
 {\widehat Z_{N,\kappa}^{K}}
 q_\sigma(x-y)a_{N,\kappa}(y,x).
\end{align*}
Thus the accepted-move part of
$\whatnu(dx)P_{N,\kappa}(x,dy)$ is symmetric in $(x,y)$. Its rejection
part is supported on the diagonal and is also symmetric. 
Hence, for
all $A,B\in\mathcal B(K)$,
\[
 \int_A P_{N,\kappa}(x,B)\,\whatnu(dx)
 =
 \int_B P_{N,\kappa}(y,A)\,\whatnu(dy).
\]
Taking $A=K$ proves that $\whatnu$ is invariant.

For $x,y\in K$, the diameter bound gives
\[
 q_\sigma(y-x)
 \ge(2\pi\sigma^2)^{-D/2}
     \exp\left\{-\frac{L_K^2}{2\sigma^2}\right\}.
\]
Moreover, $0\le\Psi_{\wideM}\le\chi(4r^2)$ on $K$, and therefore
\[
 g(y)-g(x)
 +\frac{\kappa}{2}
 \left[\Psi_{\wideM}(y)-\Psi_{\wideM}(x)\right]
 \le
 \osc_K(g)+\frac{\kappa}{2}\chi(4r^2).
\]
It follows from \eqref{eq:rwmh-acceptance} that
\[
 a_{N,\kappa}(x,y)
 \ge
 \exp\left\{
 -\osc_K(g)-\frac{\kappa}{2}\chi(4r^2)
 \right\}.
\]
Dropping the nonnegative rejection part of the transition kernel gives
\begin{align*}
 P_{N,\kappa}(x,A)
 &\ge\int_Aq_\sigma(y-x)a_{N,\kappa}(x,y)\,dy\\
 &\ge
 \underline\varepsilon_{\kappa,\sigma}
 \frac{\Leb_D(A)}{|K|},
 \qquad A\in\mathcal B(K),
\end{align*}
which proves \eqref{eq:rwmh-uniform-minorization}. The constant is
positive, and for every $x\in K$,
\[
 \underline\varepsilon_{\kappa,\sigma}
 \le
 |K|(2\pi\sigma^2)^{-D/2}
 \exp\left\{-\frac{L_K^2}{2\sigma^2}\right\}
 \le\int_Kq_\sigma(y-x)\,dy<1.
\]
The strict inequality holds because a nondegenerate Gaussian places
positive mass outside the bounded set $K$.

\paragraph{Total-variation contraction and uniqueness.}
The minorization \eqref{eq:rwmh-uniform-minorization} yields a Markov
kernel $\mathcal R$ such that
\[
 P_{N,\kappa}(x,A)
 =
 \underline\varepsilon_{\kappa,\sigma}
 \frac{\Leb_D(A)}{|K|}
 +(1-\underline\varepsilon_{\kappa,\sigma})\mathcal R(x,A).
\]
The common first term cancels when comparing the images of two
probability measures, and every Markov kernel is nonexpansive in total
variation. Therefore,
\[
 \lVert\mu_1P_{N,\kappa}-\mu_2P_{N,\kappa}\rVert_{\TV}
 \le
 (1-\underline\varepsilon_{\kappa,\sigma})
 \lVert\mu_1-\mu_2\rVert_{\TV}.
\]
Applying this inequality repeatedly with $\mu_2=\whatnu$ and using
the invariance of $\whatnu$ proves 
\[
 \lVert\rho P_{N,\kappa}^m-\whatnu\rVert_{\TV}
 \le
 (1-\underline\varepsilon_{\kappa,\sigma})^m
 \lVert\rho-\whatnu\rVert_{\TV},
\]
which is \eqref{eq:rwmh-tv-rate}. If $\widetilde\nu$ is another
invariant probability measure, the same inequality gives
\[
 \lVert\widetilde\nu-\whatnu\rVert_{\TV}
 \le
 (1-\underline\varepsilon_{\kappa,\sigma})^m
 \lVert\widetilde\nu-\whatnu\rVert_{\TV}.
\]
Letting $m\to\infty$ proves $\widetilde\nu=\whatnu$.

\paragraph{Conversion to Wasserstein distance.}
Applying \eqref{eq:w2-tv-coupling} to
$\mu_1=\rho P_{N,\kappa}^m$ and $\mu_2=\whatnu$, and then using
\eqref{eq:rwmh-tv-rate}, proves
\eqref{eq:rwmh-w2-rate}.
\end{proof}

\subsection{Proof of Theorem~\ref{thm:uniform-penalty-lower-bound}}\label{sec:prf-lowerbound}

We collect and summarise some notations introduced for convenience at various points of the proof.
\begin{center}
\begin{tabular}{|@{}cl|cl@{}|}
\toprule
Symbol & Definition & Symbol & Definition \\
\midrule
$\epsilon$ & $d_H(\mathcal X_N,M)$
& $P$ & $\operatorname{conv}(\mathcal X_N)$ \\
$A_{t,N}$ & $\operatorname{Conv}(t,\mathcal X_N)$
& $\Delta$ & $d_H(A_{t,N},M)$ \\
$s$ & $\kappa^{-1/2}$
& $w_s(u)$ & $\exp\{-\chi(u^2)/(2s^2)\}$ \\
$\mathcal Z_E(s)$ & $\displaystyle\int_K w_s(d_E(x))\,dx$
&  &  \\
\bottomrule
\end{tabular}
\end{center}

To prove this theorem, we first state and prove the following lemma.
\begin{lemma}[Quadratic tangent and convex-hull deviation]
\label{lem:convex-reach-tangent}
Let $\Omega\subset\R^2$ be a convex body whose boundary $M$ is a
$C^1$ curve of positive reach. For $y\in M$, let $n_y$ be the outward
unit normal. Then, for every $0<\tau<\Reach(M)$ and every $p\in M$,
\begin{equation}
 0\le \langle y-p,n_y\rangle
 \le \frac{\|p-y\|^2}{2\tau}.
\label{eq:lower-bound-tangent-quadratic}
\end{equation}
Moreover, if a nonempty set $S\subset M$ satisfies $\diam(S)\le h$, then
\begin{equation}
 \operatorname{conv}(S)
 \subset
 \left\{x\in\Omega:d_M(x)\le\frac{h^2}{2\tau}\right\}.
\label{eq:lower-bound-local-hull-layer}
\end{equation}
\end{lemma}

\begin{proof}
By \citet[Theorem~4.18]{federer1959curvature} (and more specifically its consequence in Page 32 of \cite{aamari2017thesis}),
\[
 d_{T_yM}(p-y)\le \frac{\|p-y\|^2}{2\tau}.
\]
In $\R^2$, the tangent line satisfies $T_yM=n_y^\perp$, and hence
\[
 d_{T_yM}(v)=|\langle v,n_y\rangle|,
 \qquad v\in\R^2.
\]
Taking $v=p-y$ and using convexity of $\Omega$ together with the outward
orientation of $n_y$ gives
\[
 d_{T_yM}(p-y)=|\langle p-y,n_y\rangle|
 =\langle y-p,n_y\rangle.
\]
Substitution into the reach inequality proves
\eqref{eq:lower-bound-tangent-quadratic}.

Now fix $y\in S$ and write an arbitrary $x\in\operatorname{conv}(S)$ as
$x=\sum_j\lambda_jp_j$, where $p_j\in S$, $\lambda_j\ge0$, and
$\sum_j\lambda_j=1$. By \eqref{eq:lower-bound-tangent-quadratic},
\[
 0\le\langle y-x,n_y\rangle
 =\sum_j\lambda_j\langle y-p_j,n_y\rangle
 \le\frac{h^2}{2\tau}.
\]
Note that the previous display is for all $y\in S$. Further note that $x\in\Omega$ due to convexity.  Consequently,
$$d_M(x)=\langle y-x,n_y\rangle\le h^2/(2\tau),$$ proving
\eqref{eq:lower-bound-local-hull-layer}.
\end{proof}

\begin{proof}
We use the body $\Omega$ displayed in the theorem. The function
$u\mapsto(u_+)^4$ is $C^3$, convex, and nondecreasing on $[0,\infty)$.
Consequently, the defining function of $\Omega$ is convex and $C^3$, and
its gradient does not vanish on the level set defining $M$. Thus $\Omega$
is a compact convex body with $C^3$ boundary. One of its straight edges is
the upper segment
\[
 \{(x_1,2):|x_1|\le2\}.
\]
In the first quadrant, its curved corner is described explicitly by
\[
 (x_1-2)^4+(x_2-1)^4=1,
 \qquad 2<x_1<3,\quad 1<x_2<2.
\]
The curvature is strictly positive in the interior of this curved corner.
Fix a closed portion $\Gamma$ away from its endpoints; continuity gives a constant
$k_*>0$ such that the curvature on $\Gamma$ is at least $k_*$. A compact
$C^3$ convex curve has positive reach, and
$\Omega\subset\overline B(0,4)$, so we may take $z_0=0$ and $R_0=4$ in
Assumption~\ref{ass:ambient-envelope}. Fix the resulting $K$, choose $r$ as
in Section~\ref{sec:model}, and fix an admissible cap $\chi$. All constants
below may depend on these fixed objects and on $\eta$, but not on $t$,
$\mathcal X_N$, or $\kappa$.

\begingroup
\renewcommand{\theHfigure}{lowerbound.\arabic{figure}}
\refstepcounter{figure}
\label{fig:lower-bound-body}
\begin{center}
\begin{tikzpicture}[scale=0.72]
  \draw[thick] (-2,2)--(2,2);
  \draw[thick] (-2,-2)--(2,-2);
  \draw[thick] (3,-1)--(3,1);
  \draw[thick] (-3,-1)--(-3,1);
  \foreach \sx/\sy in {1/1,-1/1,1/-1,-1/-1}{
    \draw[thick,samples=60,domain=0:90,smooth,variable=\a]
      plot ({\sx*(2+sqrt(cos(\a)))},{\sy*(1+sqrt(sin(\a)))});
  }
  \draw[blue!70!black,very thick] (-1.25,2)--(1.25,2);
  \draw[red!75!black,very thick,samples=40,domain=25:65,smooth,variable=\a]
    plot ({2+sqrt(cos(\a))},{1+sqrt(sin(\a))});
  \node[blue!70!black,above] at (0,2.08) {straight edge};
  \draw[red!75!black,->] (3.55,1.75)--(2.76,1.68);
  \node[red!75!black,right] at (3.5,1.75) {curved corner $\Gamma$};
  \node at (0,0) {$\Omega$};
\end{tikzpicture}
\par\smallskip
\begin{minipage}{0.9\linewidth}
\small
\textbf{Figure~\thefigure:}
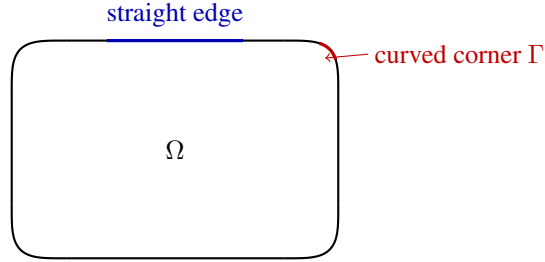
The fixed convex body in
Theorem~\ref{thm:uniform-penalty-lower-bound}. The highlighted horizontal
segment is one of the straight edges, while the highlighted portion
$\Gamma$ of a curved corner has curvature bounded below by a positive
constant.
\end{minipage}
\end{center}
\endgroup

Write
\[
 \epsilon:=d_H(\mathcal X_N,M),
 \qquad
 P:=\operatorname{conv}(\mathcal X_N).
\]
We always take $t$ sufficiently small. By
\eqref{eq:lower-bound-observation-density},
$\epsilon\le\eta t$, where $\eta<1/3$. In the first three steps set
$s:=\kappa^{-1/2}$. For $u\ge0$, put
\[
 w_s(u):=\exp\left\{-\frac{\chi(u^2)}{2s^2}\right\},
\]
and denote by $\mathcal Z_E(s)$ the normalizing constant of
$\nu_{E,s^{-2}}$ for a compact set $E\subset K$. Since $g\equiv0$, this is
the integral over $K$ of $w_s(d_E(x))$.

\paragraph{Step 1: Bounds on the Hausdorff distance.}
Fix $0<\tau<\Reach(M)$, and let $n_y$ be the outward unit normal at
$y\in M$. Lemma~\ref{lem:convex-reach-tangent} gives
\eqref{eq:lower-bound-tangent-quadratic} for every $p\in M$.

We introduce some notation for convenience. Choose the cyclic orientation--without losing generality--of $M$, and in that orientation rewrite $\Xcal_N=\{p_1,\ldots,p_N\}$. Set $p_{N+1}:=p_1$, and let $\Gamma_i$ be the
subarc from $p_i$ to $p_{i+1}$, with (arc-)length
$\ell_i$. $y_i$ divides $\Gamma_i$ into two subarcs of length $\ell_i/2$ (its midpoint). 

Choose
$q_i\in\mathcal X_N$ with $\|q_i-y_i\|\le\epsilon$. When $\epsilon$ is
smaller than a fixed fraction of $\Reach(M)$, $q_i$ lies on the same local
branch of $M$ as $y_i$. Since the interior of $\Gamma_i$ contains no
points of $\mathcal X_N$, its arc-length distance from $y_i$ is at least
$\ell_i/2$, as illustrated in Figure \ref{fig:arc}.
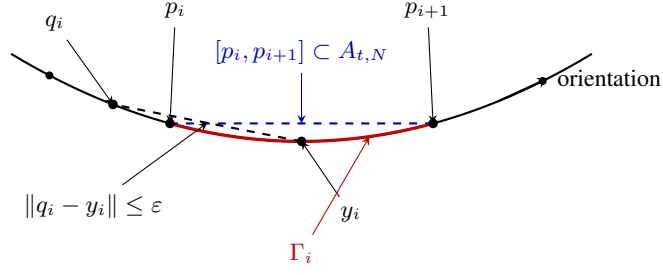
\begin{center}
\begin{tikzpicture}[x=1.45cm,y=1.65cm,>=stealth,font=\small]\label{fig:arc}
  \draw[thick,samples=80,domain=-2.65:2.65,smooth,variable=\x]
    plot ({\x},{0.4+0.1*\x*\x});
  \draw[red!75!black,very thick,samples=40,domain=-1.2:1.2,
        smooth,variable=\x]
    plot ({\x},{0.4+0.1*\x*\x});
  \coordinate (pi) at (-1.2,0.544);
  \coordinate (pip) at (1.2,0.544);
  \coordinate (yi) at (0,0.4);
  \coordinate (qi) at (-1.72,0.696);
  \coordinate (epsmid) at (-0.86,0.548);

  \draw[blue!70!black,dashed,thick] (pi)--(pip);
  \draw[dashed,thick] (yi)--(qi);
  \fill (pi) circle (1.8pt);
  \fill (pip) circle (1.8pt);
  \fill (yi) circle (1.8pt);
  \fill (qi) circle (1.8pt);
  \fill (-2.3,0.929) circle (1.5pt);
  \fill (2.2,0.884) circle (1.5pt);

  \node (qilabel) at (-2.25,1.35) {$q_i$};
  \draw[->] (qilabel)--(qi);
  \node (pilabel) at (-1.15,1.45) {$p_i$};
  \draw[->] (pilabel)--(pi);
  \node (piplabel) at (1.15,1.45) {$p_{i+1}$};
  \draw[->] (piplabel)--(pip);
  \node (yilabel) at (0.45,-0.18) {$y_i$};
  \draw[->] (yilabel)--(yi);

  \node[blue!70!black] (chordlabel) at (0,1.12)
    {$[p_i,p_{i+1}]\subset A_{t,N}$};
  \draw[blue!70!black,->] (chordlabel)--(0,0.544);
  \node[red!75!black] (arclabel) at (0,-0.5)
    {$\Gamma_i$};
  \draw[red!75!black,->] (arclabel)--(0.62,0.438);
  \node (epslabel) at (-1.9,-0.12)
    {$\|q_i-y_i\|\le\epsilon$};
  \draw[->] (epslabel)--(epsmid);

  \draw[->,thick] (1.8,0.724)--(2.25,0.906);
  \node[right] at (2.25,0.906) {orientation};
\end{tikzpicture}
\par\smallskip
\begin{minipage}{0.88\linewidth}
\small
Consecutive observations $p_i,p_{i+1}$ bound the arc $\Gamma_i$.
Hausdorff density supplies an observation $q_i$ within $\epsilon$ of its
midpoint $y_i$; hence $q_i$ lies beyond one endpoint. The opposite-side case
is symmetric.
\end{minipage}
\end{center}

For $u,v\in M$, let $d_M^{\mathrm{geo}}(u,v)$ denote the length of the shorter subarc of $M$
joining them. Apply \citet[Lemma~3]{boissonnat2019reach} with
$S=M$, $r=\Reach(M)$, $a=u$, and $b=v$. Whenever
$\|u-v\|<2\Reach(M)$, it gives
\[
 d_M^{\mathrm{geo}}(u,v)
 \le 2\Reach(M)\arcsin\left(
 \frac{\|u-v\|}{2\Reach(M)}\right).
\]
Since $\Reach(M)>0$ is fixed, the Taylor expansion of $\arcsin$ at zero
gives
\[
 \begin{aligned}
 &2\Reach(M)\arcsin\left(
   \frac{\|u-v\|}{2\Reach(M)}\right)\\
 &\qquad=\|u-v\|
 +\frac{\|u-v\|^3}{24\Reach(M)^2}
 +O(\|u-v\|^5)
 \le\|u-v\|+C\|u-v\|^2
 \end{aligned}
\]
for $u,v$ sufficiently close. Consequently,
\[
 \frac{\ell_i}{2}
 \le d_M^{\mathrm{geo}}(y_i,q_i)
 \le\epsilon+C\epsilon^2.
\]
Thus
$\|p_i-p_{i+1}\|\le\ell_i\le2\epsilon+C\epsilon^2$. The minimum
enclosing-ball radius of $\{p_i,p_{i+1}\}$ is half this distance, which is
at most
\[
 \epsilon+C\epsilon^2
 \le\eta t+C\eta^2t^2<t,
\]
whenever
$t<(1-\eta)/(C\eta^2)$. Hence every straight
chord $[p_i,p_{i+1}]$, $1\le i\le N$, belongs to $A_{t,N}$.

It follows in particular that
\begin{equation}
 \partial P\subset A_{t,N}.
\label{eq:lower-bound-polygon-boundary}
\end{equation}
We first bound $\sup_{y\in M}d_{A_{t,N}}(y)$, and then
$\sup_{x\in A_{t,N}}d_M(x)$.

For $n\in\mathbb S^1$, choose $y\in M$ maximizing
$\langle y,n\rangle$ over $\Omega$, and an observed point $p$ with
$\|p-y\|\le\epsilon$. Equation~\eqref{eq:lower-bound-tangent-quadratic}
and $p\in P\subset\Omega$ give
\[
 0\le
 \sup_{x\in\Omega}\langle x,n\rangle
 -\sup_{x\in P}\langle x,n\rangle
 \le \langle y-p,n\rangle
 \le C\epsilon^2.
\]
Note that this bound is uniform in $n$. Since $P\subset\Omega$, the support-vector formula for Hausdorff distance
(as in, for example, eq. (10) \cite{liu2023asymptotic}) yields
\begin{equation}
 d_H(P,\Omega)\le C\epsilon^2.
\label{eq:lower-bound-convex-hulls}
\end{equation}
Together with \eqref{eq:lower-bound-polygon-boundary}, this implies
$\sup_{y\in M}d_{A_{t,N}}(y)\le Ct^2$. We next control
$\sup_{x\in A_{t,N}}d_M(x)$.

For the fixed $t$, each local convex hull in the union defining
$A_{t,N}$ has the form
$\operatorname{conv}(S)$ for some nonempty $S\subseteq\mathcal X_N$ with
$\operatorname{rad}(S)\le t$, and hence $\diam(S)\le2t$. Applying
\eqref{eq:lower-bound-local-hull-layer} of
Lemma~\ref{lem:convex-reach-tangent} with $h=2t$ and taking the union over
all admissible $S$ gives
\[
 \sup_{x\in A_{t,N}}d_M(x)\le Ct^2.
\]
Combined with the bound on $\sup_{y\in M}d_{A_{t,N}}(y)$, we have shown the upper bound of \eqref{eq:lower-bound-basic-geometry-main}
\[
 d_H(A_{t,N},M)\le Ct^2.
\]

We turn to the lower bound. 

We again require some notation. Recall from
Figure~\ref{fig:lower-bound-body} the fixed closed portion $\Gamma$ of a
curved corner on which the curvature is at least $k_*$. Choose points
$z_1,\ldots,z_m$ along a fixed portion contained in the interior of
$\Gamma$, with
consecutive centers separated by arclength $4t$.
Since $\Gamma$ has fixed positive length, $m\ge c/t$ for small $t$, the
balls $B(z_j,t)$ are pairwise disjoint. For each $j$, let $z_j^-$ and
$z_j^+$ be the boundary points on either side of $z_j$ satisfying
\[
 \|z_j^--z_j\|=\|z_j^+-z_j\|=(1-\eta)t.
\]
By \eqref{eq:lower-bound-observation-density} and the definition of
$\epsilon$, there exist observations
$p_j^-,p_j^0,p_j^+\in\mathcal X_N$ within distance $\epsilon$ of
$z_j^-,z_j,z_j^+$, respectively. Since $\epsilon\le\eta t$, all three
observations lie in $B(z_j,t)$, so their triangle
\[
 T_j:=\operatorname{conv}\{p_j^-,p_j^0,p_j^+\}
\]
is contained in $A_{t,N}$. See Figure \ref{fig:lowerboundcons} for illustrative reference.

\begin{center}
\begin{tikzpicture}[x=0.88cm,y=0.88cm,>=stealth,font=\scriptsize]\label{fig:lowerboundcons}
  \begin{scope}[xshift=-3.6cm]
    \node[font=\small] at (0,2.15) {one local triangle};
    \draw[thick,samples=60,domain=-2.2:2.2,smooth,variable=\x]
      plot ({\x},{0.28*\x*\x});
    \coordinate (zj) at (0,0);
    \coordinate (zm) at (-1.35,0.510);
    \coordinate (zp) at (1.35,0.510);
    \coordinate (pm) at (-1.10,0.339);
    \coordinate (p0) at (0.12,0.004);
    \coordinate (pp) at (1.12,0.351);

    \draw[dashed,gray] (zj) circle (1.55);
    \node[gray] at (-1.55,-1.25) {$B(z_j,t)$};
    \fill[blue!25,opacity=0.75] (pm)--(p0)--(pp)--cycle;
    \draw[blue!70!black,thick] (pm)--(p0)--(pp)--cycle;
    \node[blue!70!black] at (0.06,0.22) {$T_j$};

    \draw[dotted,gray,thick] (zm)--(pm) (zj)--(p0) (zp)--(pp);
    \draw (zm) circle (2pt);
    \draw (zj) circle (2pt);
    \draw (zp) circle (2pt);
    \fill (pm) circle (1.8pt);
    \fill (p0) circle (1.8pt);
    \fill (pp) circle (1.8pt);

    \node at (-1.68,0.90) {$z_j^-$};
    \node at (1.68,0.90) {$z_j^+$};
    \node at (-0.38,-0.42) {$z_j$};
    \node at (-1.50,0.12) {$p_j^-$};
    \node at (1.50,0.12) {$p_j^+$};
    \node at (0.43,-0.42) {$p_j^0$};
  \end{scope}

  \begin{scope}[xshift=3.6cm]
    \node[font=\small] at (0,2.15) {packing along $\Gamma$};
    \draw[red!75!black,very thick,samples=70,domain=-3.0:3.0,
          smooth,variable=\x]
      plot ({\x},{0.10*\x*\x});
    \foreach \x/\lab in {-2/{z_{j-1}},0/{z_j},2/{z_{j+1}}}{
      \pgfmathsetmacro{\yy}{0.10*\x*\x}
      \draw[dashed,gray] (\x,\yy) circle (0.72);
      \fill (\x,\yy) circle (1.8pt) node[below] {$\lab$};
    }
    \draw[<->] (-2,1.35)--(0,1.35)
      node[midway,above] {arclength $4t$};
    \node[align=center] at (0,-1.15)
      {$m\ge c/t$ disjoint\\radius-$t$ neighborhoods};
  \end{scope}
\end{tikzpicture}
\par\smallskip
\begin{minipage}{0.9\linewidth}
\small
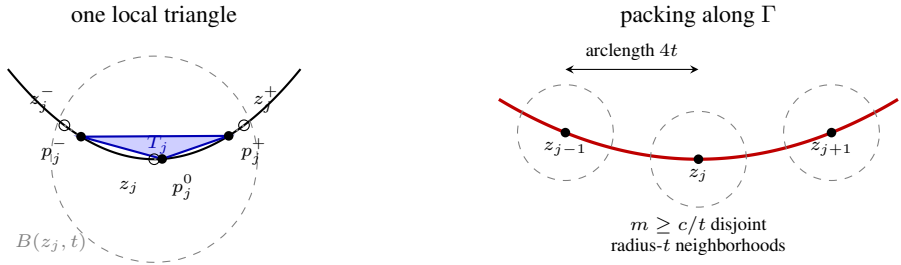
Open circles are the reference points $z_j^-,z_j,z_j^+$ on $\Gamma$;
filled circles are the corresponding observations, and each dotted link has
length at most $\epsilon$. Curvature gives the shaded triangle base of order
$t$ and height of order $t^2$.
\end{minipage}
\end{center}
The previous argument for the upper bound was via a carefully chosen covering over the point cloud, now to derive the lower bound we will consider a carefully chosen packing.
Set
\[
 a_j:=d_M^{\mathrm{geo}}(p_j^-,p_j^0),
 \qquad
 b_j:=d_M^{\mathrm{geo}}(p_j^0,p_j^+).
\]
The local arclength estimate above and $\epsilon\le\eta t$ give
\[
 a_j,b_j\ge(1-3\eta)t-Ct^2.
\]
Thus, since $\eta<1/3$, $c_\eta t\le a_j,b_j\le Ct$ for small $t$.

Let $\gamma_j$ be the unit-speed parametrization of $M$ with
$\gamma_j(0)=p_j^0$, $\gamma_j(-a_j)=p_j^-$, and
$\gamma_j(b_j)=p_j^+$. The local expansion
\citep[Lemma~2(iv)]{aamari2019nonasymptotic} gives
\[
 \gamma_j(u)=p_j^0+u\gamma_j'(0)+\frac{u^2}{2}\gamma_j''(0)+R_j(u),
 \qquad
 \|R_j(u)\|\le C|u|^3,
\]
Since $\gamma_j$ has unit speed, $\gamma_j'(0)$ is its unit tangent, while $\gamma_j''(0)$ is perpendicular to it and has magnitude equal to the curvature.
and $\|\gamma_j''(0)\|$ is the curvature at $p_j^0$, hence is at least
$k_*$ on $\Gamma$. Therefore, with $\operatorname{det}$ to be the usual determinant of a square matrix,
\[
 \begin{aligned}
 2\Leb_2(T_j)
 &=\left|\det\bigl(
     p_j^- -p_j^0,
     p_j^+ -p_j^0
   \bigr)\right|\\
 &=\frac{\|\gamma_j''(0)\|}{2}a_jb_j(a_j+b_j)+O(t^4)
 \ge c(k_*,\eta)t^3
 \end{aligned}
\]
for small $t$. Since the containing balls are disjoint, so are the
triangles. Thus
$\Leb_2(A_{t,N})\ge ct^2$. The boundary-layer inclusion proved above gives
$\Leb_2(A_{t,N})\le Ct^2$. Since a boundary layer of width
$u$ around the fixed curve has area at most $Cu$, the area lower bound also
forces $\sup_{x\in A_{t,N}}d_M(x)\ge ct^2$. We have proved 
\begin{equation}
 ct^2\le \underbrace{d_H(A_{t,N},M)}_{=\Delta}\le Ct^2, \qquad ct^2\le \Leb_2(A_{t,N})\le Ct^2 
\label{eq:lower-bound-basic-geometry}
\end{equation}
which shows \eqref{eq:lower-bound-basic-geometry-main}.

\paragraph{Step 2: comparison of the normalizing constants.}

Recall that $\Zcal$ denotes the normalizing constants. We will show
\begin{equation*}
 ct^2\le\mathcal Z_{A_{t,N}}(s)-\mathcal Z_M(s)\le Ct^2,
 \qquad
 \mathcal Z_M(s)\asymp s,
 \qquad
 \mathcal Z_{A_{t,N}}(s)\asymp s+t^2
\end{equation*}
starting with the first bound.

Since $\chi(v)=\chi(4r^2)$ for $v\ge4r^2$, the definition of $w_s$ gives
$w_s(u)=w_s(2r)$ whenever $u\ge2r$. Therefore, for every $x\in K$ and any set $E \subseteq K$,
\[
 w_s(d_E(x))
 =w_s(2r)
 +\mathbf 1_{\{d_E(x)<2r\}}
  \{w_s(d_E(x))-w_s(2r)\}.
\]
Integrating this pointwise identity over $K$ gives
\[
 \mathcal Z_E(s)
 =w_s(2r)\Leb_2(K)
 +\int_{\{x:d_E(x)<2r\}}
   \{w_s(d_E(x))-w_s(2r)\}\,dx.
\]
The set $\{d_E<2r\}$ is the disjoint union of
$E=\{d_E=0\}$ and $\{0<d_E<2r\}$. On $E$, $w_s(d_E)=w_s(0)=1$.
On the remaining set, $\lVert\nabla d_E\rVert=1$ almost everywhere, so
the coarea formula \citep[Theorem 3.2.22]{federer1969geometric} replaces integration over $x$ by integration over the
curves $\{d_E=u\}$. Thus
\begin{align}
 \mathcal Z_E(s)=&w_s(2r)\Leb_2(K)
 +\int_{\{x:d_E(x)=0\}}
   \{w_s(d_E(x))-w_s(2r)\}\,dx \\
  &+ \int_{\{x:d_E(x)\in(0,2r)\}}
   \{w_s(d_E(x))-w_s(2r)\}\,dx  \\
 =&w_s(2r)\Leb_2(K)
 +(1-w_s(2r))\Leb_2(E)\notag\\
 &+\int_0^{2r}\{w_s(u)-w_s(2r)\}
   \operatorname{length}\bigl(\{x\in K:d_E(x)=u\}\bigr)\,du.
\label{eq:lower-bound-baseline-decomposition}
\end{align}

For $E=M$, recall from \eqref{eq:volume-def} that
$V_M=\Vol_M(M)$, which is the length of $M$ here. At distance $u<2r$
from $M$, the curve outside $\Omega$ has length $V_M+2\pi u$, while the
curve inside $\Omega$ has length $V_M-2\pi u$. Their lengths add to $2V_M$.
Since $M$ has zero planar area,
\eqref{eq:lower-bound-baseline-decomposition} becomes
\begin{equation}
 \mathcal Z_M(s)
 =w_s(2r)\Leb_2(K)
  +2V_M\int_0^{2r}\{w_s(u)-w_s(2r)\}\,du,
\label{eq:lower-bound-manifold-normalizer}
\end{equation}

Taking $E=A_{t,N}$ in
\eqref{eq:lower-bound-baseline-decomposition} gives
\begin{align}
 \mathcal Z_{A_{t,N}}(s)
 ={}&w_s(2r)\Leb_2(K)
 +(1-w_s(2r))\Leb_2(A_{t,N})\notag\\
 &+\int_0^{2r}\{w_s(u)-w_s(2r)\}
 \operatorname{length}\bigl(\{x\in K:d_{A_{t,N}}(x)=u\}\bigr)\,du.
\label{eq:lower-bound-reconstruction-normalizer}
\end{align}
Subtracting \eqref{eq:lower-bound-manifold-normalizer} from
\eqref{eq:lower-bound-reconstruction-normalizer}, the common term $w_s(2r)\Leb_2(K)$ is cancelled, and we get
\begin{align}
 \mathcal Z_{A_{t,N}}(s)-\mathcal Z_M(s)
 ={}&(1-w_s(2r))\Leb_2(A_{t,N})\notag\\
 &+\int_0^{2r}\{w_s(u)-w_s(2r)\}\notag\\
 &\qquad\times
 \Bigl\{\operatorname{length}
 \bigl(\{x\in K:d_{A_{t,N}}(x)=u\}\bigr)-2V_M\Bigr\}\,du.
\label{eq:lower-bound-normalizer-exact-difference}
\end{align}

Recall from \eqref{eq:lower-bound-basic-geometry} that $\exists \,c_A,C_A>0$ such that
\[
 c_At^2\le\Leb_2(A_{t,N})\le C_At^2.
\]
Recall from \eqref{eq:cap-inner}--\eqref{eq:cap-outer} that $\chi(u^2)=u^2$ for $0\leq u\leq r$, which in turn implies
\[
 w_s(u)=e^{-u^2/(2s^2)},
 \qquad 0\le u\le r.
\]
Thus, $\exists\, s_0>0 : \forall\,0<s\le s_0$,
\[
 cs\leq s\int_0^{r/s}e^{-v^2/2}\,dv\leq Cs.
\]

Since
$ \int_r^{2r} w_s(u)\,du$
is negligible, we have
\[
 cs\leq s\int_0^{2r}e^{-v^2/2}\,dv\leq Cs.
\]
By similar arguments
\[
 w_s(2r)\le\frac12
 \qquad\text{and}\qquad
 w_s(2r)\Leb_2(K)\le Cs,
\]
which gives us
\begin{equation}
 c s\le\int_0^{2r}\{w_s(u)-w_s(2r)\}\,du\le Cs,
 \qquad
 w_s(2r)\le\frac12,
 \qquad
 w_s(2r)\Leb_2(K)\le Cs.
\label{eq:lower-bound-weight-integral}
\end{equation}
 It remains only to prove, uniformly for $0<u\le2r$, that
\begin{equation}
 -Ct^2\le
 \operatorname{length}
 \bigl(\{x\in K:d_{A_{t,N}}(x)=u\}\bigr)-2V_M
 \le0
\label{eq:lower-bound-length-task}
\end{equation}
which would place the integral in
\eqref{eq:lower-bound-normalizer-exact-difference} between $-Cst^2$ and
$0$. Figure~\ref{fig:lower-bound-level-sets} illustrates the two lengths
being compared in this bound.

\begingroup
\renewcommand{\theHfigure}{lowerboundlevels.\arabic{figure}}
\refstepcounter{figure}
\label{fig:lower-bound-level-sets}
\begin{center}
\begin{tikzpicture}[x=0.75cm,y=0.75cm,>=stealth,font=\scriptsize]
  \begin{scope}[xshift=-3.8cm]
    \node[font=\small] at (0,2.45) {true constraint};
    \draw[black,very thick]
      plot[smooth cycle,tension=0.7] coordinates
      {(-2.15,0) (-1.75,1.35) (0,1.55) (1.75,1.35)
       (2.15,0) (1.75,-1.35) (0,-1.55) (-1.75,-1.35)};
    \draw[red!75!black,dashed,thick]
      plot[smooth cycle,tension=0.7] coordinates
      {(-2.48,0) (-2.02,1.63) (0,1.86) (2.02,1.63)
       (2.48,0) (2.02,-1.63) (0,-1.86) (-2.02,-1.63)};
    \draw[red!75!black,dashed,thick]
      plot[smooth cycle,tension=0.7] coordinates
      {(-1.80,0) (-1.47,1.07) (0,1.24) (1.47,1.07)
       (1.80,0) (1.47,-1.07) (0,-1.24) (-1.47,-1.07)};
    \node at (1.62,1.47) {$M$};
    \draw[<->,gray!80,thick] (2.12,0.58)--(2.43,0.67)
      node[midway,right] {$u$};
    \draw[<->,gray!80,thick] (1.78,-0.56)--(2.09,-0.65)
      node[midway,right] {$u$};
    \node[red!75!black] at (0,-2.10) {$\{d_M=u\}$};
    \node[align=center] at (0,-2.65)
      {inner and outer lengths sum to $2V_M$};
  \end{scope}

  \begin{scope}[xshift=3.8cm]
    \node[font=\small] at (0,2.45) {reconstructed constraint};
    \draw[gray!75,very thick]
      plot[smooth cycle,tension=0.7] coordinates
      {(-2.15,0) (-1.75,1.35) (0,1.55) (1.75,1.35)
       (2.15,0) (1.75,-1.35) (0,-1.55) (-1.75,-1.35)};
    \fill[blue!22,even odd rule]
      (-1.95,0.92)--(-0.85,1.38)--(0.62,1.31)--(1.88,0.93)--
      (2.03,-0.18)--(1.72,-1.18)--(0.28,-1.43)--(-1.58,-1.12)--
      (-1.96,-0.16)--cycle
      (-1.33,0.52)--(-0.52,0.91)--(0.78,0.73)--(1.38,0.50)--
      (1.24,-0.47)--(0.35,-0.82)--(-0.78,-0.70)--(-1.35,-0.34)--cycle;
    \draw[blue!75!black,very thick]
      (-1.95,0.92)--(-0.85,1.38)--(0.62,1.31)--(1.88,0.93)--
      (2.03,-0.18)--(1.72,-1.18)--(0.28,-1.43)--(-1.58,-1.12)--
      (-1.96,-0.16)--cycle;
    \draw[blue!75!black,very thick]
      (-1.33,0.52)--(-0.52,0.91)--(0.78,0.73)--(1.38,0.50)--
      (1.24,-0.47)--(0.35,-0.82)--(-0.78,-0.70)--(-1.35,-0.34)--cycle;
    \draw[red!75!black,dashed,thick]
      plot[smooth cycle,tension=0.65] coordinates
      {(-2.30,0) (-1.92,1.45) (0,1.72) (1.98,1.42)
       (2.34,0) (1.93,-1.46) (0,-1.70) (-1.95,-1.43)};
    \draw[red!75!black,dashed,thick]
      plot[smooth cycle,tension=0.65] coordinates
      {(-0.98,0) (-0.72,0.55) (0,0.64) (0.88,0.48)
       (0.98,0) (0.68,-0.50) (0,-0.59) (-0.82,-0.46)};
    \node[blue!75!black] at (0,1.12) {$A_{t,N}$};
    \node[gray!75] at (1.70,1.50) {$M$};
    \node[red!75!black,align=center] at (0,-2.10)
      {$\{d_{A_{t,N}}=u\}$};
    \node[align=center] at (0,-2.65)
      {total level-set length compared with $2V_M$};
  \end{scope}
\end{tikzpicture}
\par\smallskip
\begin{minipage}{0.9\linewidth}
\small
\textbf{Figure~\thefigure:}
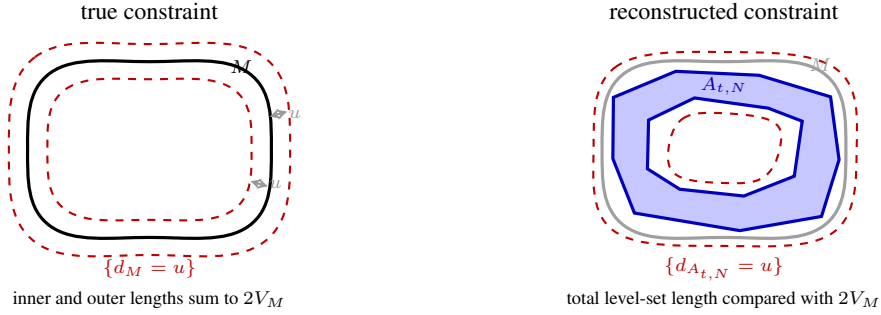
The level-set comparison in \eqref{eq:lower-bound-length-task}. The red
dashed curves are the two components of the distance-$u$ level set. The
gray curve in the right panel shows the true constraint $M$. The drawing
is schematic.
\end{minipage}
\end{center}
\endgroup

We now prove this bound.

Let $S\subset\mathcal X_N$ satisfy $\operatorname{rad}(S)\le t$. $S$ is allowed to be a singleton. Thus
for some $z_S\in\R^2$,
$S\subseteq\overline B(z_S,t)$. If $S$ is a singleton, its convex hull is
already contained in $\partial P$ by
\eqref{eq:lower-bound-polygon-boundary}, so suppose that $|S|\ge2$.
Let $a,b\in S$ be the endpoints of the
short arc of $M$ containing $S$, and denote this arc by $\Gamma_{a,b}$. Due to positive reach, for small
$t$, the tangent direction changes by less than $\pi/2$ along
$\Gamma_{a,b}$, so $\Gamma_{a,b}$ lies in the closed disk with diameter
$[a,b]$. This disk is centered at $(a+b)/2$ and has radius
$\|a-b\|/2\le t$. Consequently,
$\mathcal X_N\cap\Gamma_{a,b}$ is also contained in a radius-$t$ ball and
is an admissible subset in the definition of $A_{t,N}$. Let $H_{a,b}^+$ be the closed half-plane bounded by the line through
$a,b$ that contains $\Gamma_{a,b}$, and let $H_{a,b}^-$ be the opposite
closed half-plane. Figure~\ref{fig:lower-bound-local-cap} summarizes this construction.
\begingroup
\renewcommand{\theHfigure}{lowerboundcap.\arabic{figure}}
\refstepcounter{figure}
\label{fig:lower-bound-local-cap}
\begin{center}
\begin{tikzpicture}[x=1.05cm,y=1.05cm,>=stealth,font=\small]
  \fill[orange!7] (-2.7,0) rectangle (2.7,2.05);
  \fill[gray!8] (-2.7,-1.85) rectangle (2.7,0);
  \draw[gray!65,densely dotted] (-2.7,0)--(2.7,0);
  \fill[blue!8] (0,0) circle (1.6);
  \draw[blue!55!black,dashed,thick] (0,0) circle (1.6);
  \coordinate (zs) at (-0.18,0.16);
  \draw[gray!75,dashed] (zs) circle (1.95);

  \draw[gray!70,thick,samples=30,domain=-2.35:-1.6,smooth,variable=\x]
    plot ({\x},{-0.38*(\x+1.6)*(\x+1.6)});
  \draw[gray!70,thick,samples=30,domain=1.6:2.35,smooth,variable=\x]
    plot ({\x},{-0.38*(\x-1.6)*(\x-1.6)});
  \draw[red!75!black,very thick,samples=60,domain=-1.6:1.6,
        smooth,variable=\x]
    plot ({\x},{0.55*(1-(\x/1.6)^2)});
  \draw[blue!75!black,very thick] (-1.6,0)--(1.6,0);

  \fill (-1.6,0) circle (1.8pt) node[below left] {$a$};
  \fill (1.6,0) circle (1.8pt) node[below right] {$b$};
  \foreach \x/\y in {-0.82/0.405,0/0.55,0.78/0.419}
    \fill (\x,\y) circle (1.7pt);
  \node at (0.22,0.80) {$S\subset\Gamma_{a,b}$};
  \node[red!75!black] at (0,1.17) {$\Gamma_{a,b}$};
  \node[blue!75!black,below] at (-0.82,-0.04) {$[a,b]$};

  \fill (0,0) circle (1.3pt);

  \fill (zs) circle (1.3pt);
  \node[gray!75,above right] at (zs) {$z_S$};
  \node[gray!75] at (-2.15,1.35) {$\overline B(z_S,t)$};
  \node[blue!55!black,align=center] at (2.55,1.05)
    {disk with\\diameter $[a,b]$};

  \node[orange!70!black] at (-2.25,1.72) {$H_{a,b}^+$};
  \node[gray!75!black] at (-2.25,-1.55) {$H_{a,b}^-$};
\end{tikzpicture}
\par\smallskip
\begin{minipage}{0.88\linewidth}
\small
\textbf{Figure~\thefigure:}
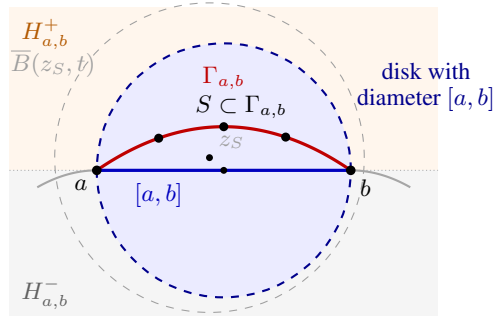
The geometry of a non-singleton admissible set $S$. The points of $S$ lie
in $\overline B(z_S,t)$. Its endpoints $a,b$ determine the short arc
$\Gamma_{a,b}$ and the chord $[a,b]$. For small $t$, the arc lies in the
disk with diameter $[a,b]$, whose radius is at most $t$. The chord line
separates $H_{a,b}^+$ from $H_{a,b}^-$.
\end{minipage}
\end{center}
\endgroup

As illustrated in Figure~\ref{fig:lower-bound-local-cap}, the line through
$a$ and $b$ cuts from $P$ the polygonal cap
\[
 C_{a,b}:=\operatorname{conv}(\mathcal X_N\cap\Gamma_{a,b})
          =P\cap H_{a,b}^+.
\]
Since $S\subseteq\mathcal X_N\cap\Gamma_{a,b}$, we have
$\operatorname{conv}(S)\subseteq C_{a,b}$.
Moreover, $C_{a,b}$ is itself one of the local convex hulls defining
$A_{t,N}$. Define the finite collection
\[
 \mathcal C:=\left\{(a,b)\in\mathcal X_N^2:
 \begin{array}{l}
 \exists \, S\subseteq\mathcal X_N : |S|\ge2, \operatorname{rad}(S)\le t,
 \text{and }a,b\in S\text{ are the endpoints of}\\ \text{the short arc of }M
 \text{ containing }S
 \end{array}\right\}.
\]
Then
\begin{equation}
 A_{t,N}=\bigcup_{(a,b)\in\mathcal C}C_{a,b}.
\label{eq:lower-bound-union-of-caps}
\end{equation}
Let
\[
 Q:=P\cap\bigcap_{(a,b)\in\mathcal C}H_{a,b}^-.
\]
By \eqref{eq:lower-bound-polygon-boundary}, no point of
$\partial P$ belongs to $P\setminus A_{t,N}$. De Morgan's identity applied
to \eqref{eq:lower-bound-union-of-caps} therefore gives
\begin{equation}
 A_{t,N}=P\setminus\operatorname{int}(Q).
\label{eq:lower-bound-cap-decomposition}
\end{equation}
Since $A_{t,N}\subseteq P$ contains $\partial P$, a point outside $P$
first meets $A_{t,N}$ at $\partial P$; since
$A_{t,N}=P\setminus\operatorname{int}(Q)$, a point in $Q$ first meets it
at $\partial Q$. Hence
\[
 d_{A_{t,N}}(x)=d_{\partial P}(x)\quad(x\notin P),
 \qquad
 d_{A_{t,N}}(x)=d_{\partial Q}(x)\quad(x\in Q),
\]
and $d_{A_{t,N}}(x)=0$ for $x\in A_{t,N}$.

The observations in $\Gamma_{a,b}$ lie in the disk with diameter $[a,b]$
and ($\because \|a-b\|/2\leq t$) have diameter at most $2t$. Applying
\eqref{eq:lower-bound-local-hull-layer} to
$S=\mathcal X_N\cap\Gamma_{a,b}$ and using the upper bound in
\eqref{eq:lower-bound-convex-hulls}, without loss of generality let $C_0$
be large enough that
\[
 C_{a,b}\subset
 \{x\in\Omega:d_M(x)<C_0t^2\},
 \qquad (a,b)\in\mathcal C,
 \qquad d_H(P,\Omega)\le C_0t^2.
\]
For a planar convex body $B$, we can shrink it by $u$ as
\[
 B_{-u}:=\{x\in B:d_{\partial B}(x)\ge u\}.
\]
Then
$\Omega_{-C_0t^2}\subset P$. 

Observe from Figure \ref{fig:lower-bound-local-cap} that if we shrink $\Omega$ by $C_0t^2$, the $\Gamma_{a,b}$ arc falls entirely within $H_{a,b}^-$. Since this is for any such $(a,b)\in\Ccal$, 
\begin{equation}
 \Omega_{-C_0t^2}\subset Q\subset P\subset\Omega.
\label{eq:lower-bound-core-inclusions}
\end{equation}
Moreover,
$(\Omega_{-C_0t^2})_{-u}=\Omega_{-(u+C_0t^2)}$. Applying these facts to
\eqref{eq:lower-bound-core-inclusions} gives
\begin{equation}
 \Omega_{-(u+C_0t^2)}
 \subset Q_{-u}
 \subset\Omega_{-u},
\qquad 0\le u\le2r
\label{eq:lower-bound-inner-parallels}
\end{equation}
where the second inclusion uses the fact $(E_1)_{-u}\subseteq(E_2)_{-u}$ for  $E_1\subseteq E_2$.

We can consider smaller $t_0$ in the hypothesis of Theorem \ref{thm:uniform-penalty-lower-bound}, so that $2r+C_0t^2<\Reach(M)$. Then, projections are unique up to distance $2r+C_0t^2$. Consequently, shrinking only reduces the curved corners of Figure \ref{fig:lower-bound-body} while keeping the straight edges unchanged in length. Therefore 
\[
 \operatorname{perimeter}(\Omega_{-v})=V_M-2\pi v,
 \qquad 0\le v\le2r+C_0t^2.
\]
Since all bodies under consideration are convex, \eqref{eq:lower-bound-core-inclusions} and
\eqref{eq:lower-bound-inner-parallels} imply, for $0\le u\le2r$,
\begin{equation}
\begin{aligned}
 V_M-2\pi(u+C_0t^2)
 &\le\operatorname{perimeter}(Q_{-u})
 \le V_M-2\pi u,\\
 V_M-2\pi C_0t^2
 &\le\operatorname{perimeter}(P)\le V_M.
\end{aligned}
\label{eq:lower-bound-perimeter-bounds}
\end{equation}
Enlarging the convex polygon $P$ by distance $v$ preserves the total
length of its translated edges and adds circular arcs of radius $v$ whose
angles sum to $2\pi$. Its perimeter therefore becomes $$\operatorname{perimeter}(P)+2\pi v.$$

The distance identities following
\eqref{eq:lower-bound-cap-decomposition} now show that, for
$0<u\le2r$, the level set outside $P$ has length
$\operatorname{perimeter}(P)+2\pi u$, while the level set inside $Q$ has
length $\operatorname{perimeter}(Q_{-u})$. Therefore
\[
 \operatorname{length}
 \bigl(\{x\in K:d_{A_{t,N}}(x)=u\}\bigr)
 =\operatorname{perimeter}(P)+2\pi u
  +\operatorname{perimeter}(Q_{-u}).
\]
Adding the bounds in \eqref{eq:lower-bound-perimeter-bounds} gives
\[
 2V_M-4\pi C_0t^2
 \le \operatorname{length}
 \bigl(\{x\in K:d_{A_{t,N}}(x)=u\}\bigr)
 \le2V_M,
 \qquad 0<u\le2r.
\]
This proves \eqref{eq:lower-bound-length-task}. Substituting this bound into
\eqref{eq:lower-bound-normalizer-exact-difference} and using
\eqref{eq:lower-bound-weight-integral} gives
\[
 (1-w_s(2r))\Leb_2(A_{t,N})-Cst^2
 \le\mathcal Z_{A_{t,N}}(s)-\mathcal Z_M(s)
 \le(1-w_s(2r))\Leb_2(A_{t,N}).
\]
Decrease $s_0$ further so that $Cs_0\le c_A/4$. Then
\[
 \frac{c_A}{4}t^2
 \le\mathcal Z_{A_{t,N}}(s)-\mathcal Z_M(s)
 \le C_At^2.
\]
Equation~\eqref{eq:lower-bound-manifold-normalizer} and
\eqref{eq:lower-bound-weight-integral} then give
\begin{equation}
 ct^2\le\mathcal Z_{A_{t,N}}(s)-\mathcal Z_M(s)\le Ct^2,
 \qquad
 \mathcal Z_M(s)\asymp s,
 \qquad
 \mathcal Z_{A_{t,N}}(s)\asymp s+t^2
\label{eq:lower-bound-normalizer-comparison}
\end{equation}
for $0<s\le s_0$. This completes our second step.

\paragraph{Step 3: two Wasserstein lower bounds.}
We now prove
\[
 W_2(\nu_{A_{t,N},s^{-2}},\pi_M)
 \ge c\frac{t^2}{s+t^2},
 \qquad
 W_2^2(\nu_{A_{t,N},s^{-2}},\pi_M)
 \ge c\frac{s^3}{s+t^2}
\]
using 
\eqref{eq:lower-bound-normalizer-comparison} to control the normalizing
constant in the denominator. We start with the first term. Since
\[
  W_2(\nu_{A_{t,N},s^{-2}},\pi_M)
 \ge W_1(\nu_{A_{t,N},s^{-2}},\pi_M)
 \]
it is sufficient to bound the Wasserstein-1 distance. By
Kantorovich--Rubinstein duality,
\begin{align*}
  W_1(\nu_{A_{t,N},s^{-2}},\pi_M)
 & =\sup_{\substack{f:\, \operatorname{Lip}(f)\le1}}
 \left|
 \int_K f(x)\,\nu_{A_{t,N},s^{-2}}(dx)
 -\int_M f(y)\,\pi_M(dy)
 \right|\\
 &\geq \left|
 \int_K f(x)\,\nu_{A_{t,N},s^{-2}}(dx)
 -\int_M f(y)\,\pi_M(dy)
 \right|\qquad \forall\, f:\operatorname{Lip}(f)=1.
\end{align*}
To get the lower bound, it will be thus sufficient to construct such an $f$, which is what we now do.

Without loss of generality, let the top straight edge of $\Omega$ lie on
the $x_1$-axis. Let $r$ be the fixed tubular-neighborhood
radius from Section~\ref{sec:model}, and choose a fixed $\rho<r/2$ such that
\[
 [-8\rho,8\rho]\times\{0\}\subset M,
\]
and no point of $M$ outside this straight edge lies in
$[-8\rho,8\rho]\times[-2\rho,2\rho]$. 
Our $f$ shall be $\rho/2 F$ where $F:\R^2\to\R$ is
\[
 F(x):=\left(1-\frac{|x_1|}{\rho}\right)_+
       \left(1-\frac{|x_2|}{\rho}\right)_+.
\]
The positive parts make $F$ vanish outside
$[-\rho,\rho]\times[-\rho,\rho]$. Thus all integrals involving $F$ are
restricted to the rectangle where
\eqref{eq:lower-bound-flat-identity} holds. Moreover, $F$ is
$2/\rho$-Lipschitz. Since $\pi_M$ is uniform arclength, we can directly evaluate the following integral
\[
 \int F(x)\pi_M(x)dx=\frac{\rho}{V_M}.
\]
Next, we evaluate 
\[
  \left|
 \int_K \frac{\rho F(x)}{2}\,\nu_{A_{t,N},s^{-2}}(dx)
 -\int_M \frac{\rho F(x)}{2}\,\pi_M(dy)\numberthis\label{eq:KR-lower-bound}
 \right|
 \]
Observe that whenever $x\in [-\rho,\rho]\times[-\rho,\rho]$,
\begin{equation}
 d_{A_{t,N}}(x)=d_M(x)=|x_2|,
\label{eq:lower-bound-flat-identity}
\end{equation}
where $x=(x_1,x_2)$.




Recall that
\[
 w_s(u)=\exp\left\{-\frac{\chi(u^2)}{2s^2}\right\},
 \qquad
 \nu_{A_{t,N},s^{-2}}(dx)
 =\frac{w_s(d_{A_{t,N}}(x))}{\mathcal Z_{A_{t,N}}(s)}\,dx.
\]
Using \eqref{eq:lower-bound-flat-identity} on the support of $F$ gives
\[
 \int_K F(x)w_s(d_{A_{t,N}}(x))\,dx
 =2\rho\int_0^\rho
   \left(1-\frac{u}{\rho}\right)w_s(u)\,du
 \le2\rho\int_0^\rho w_s(u)\,du.
\]
For $0<u<\rho$, the curve denoting the inner shrinking and outer fattening of $M$ at distance $u$
have total length $2V_M$, as established in the derivation of
\eqref{eq:lower-bound-manifold-normalizer}. Integrating over $u=d_M(x)$ gives
\[
 \mathcal Z_M(s)
 \ge \int_{\{x:d_M(x)<\rho\}}w_s(d_M(x))\,dx
 =2V_M\int_0^\rho w_s(u)\,du.
\]
Therefore
\begin{align*}
 \int_K F(x)\,\nu_{A_{t,N},s^{-2}}(dx)
 &\le \frac{\rho}{V_M}
       \frac{\mathcal Z_M(s)}{\mathcal Z_{A_{t,N}}(s)},\qquad \text{and}\qquad 
 \int_M F(y)\,\pi_M(dy)=\frac{\rho}{V_M}.
\end{align*}
We now substitute everything into \eqref{eq:KR-lower-bound} to get
\begin{align*}
 W_1(\nu_{A_{t,N},s^{-2}},\pi_M)
 &\ge \frac{\rho}{2}
 \left\{
 \int_M F(y)\,\pi_M(dy)
 -\int_K F(x)\,\nu_{A_{t,N},s^{-2}}(dx)
 \right\}\\
 &\ge\frac{\rho^2}{2V_M}
 \frac{\mathcal Z_{A_{t,N}}(s)-\mathcal Z_M(s)}
 {\mathcal Z_{A_{t,N}}(s)}.
\end{align*}
By \eqref{eq:lower-bound-normalizer-comparison},  $$\frac{\mathcal Z_{A_{t,N}}(s)-\mathcal Z_M(s)}
 {\mathcal Z_{A_{t,N}}(s)}\geq ct^2/(s+t^2).$$ Thus,
 \begin{align*}
   W_1(\nu_{A_{t,N},s^{-2}},\pi_M)
 \ge c\frac{t^2}{s+t^2},
 \qquad 0<s\le s_0,
 \end{align*}
 which in turn implies
\begin{equation}
 W_2(\nu_{A_{t,N},s^{-2}},\pi_M)
 \ge W_1(\nu_{A_{t,N},s^{-2}},\pi_M)
 \ge c\frac{t^2}{s+t^2},
 \qquad 0<s\le s_0.
\label{eq:lower-bound-tangential-deficit}
\end{equation}
We now lower bound $W_2^2$.

For any coupling of $x\sim\nu_{A_{t,N},s^{-2}}$ and $y\sim\pi_M$,
$\lVert x-y\rVert^2\ge d_M(x)^2$. Taking the infimum over all couplings
therefore gives
\[
 W_2^2(\nu_{A_{t,N},s^{-2}},\pi_M)
 \ge \frac{1}{\mathcal Z_{A_{t,N}}(s)}
 \int_K d_M(x)^2w_s(d_{A_{t,N}}(x))\,dx.
\]
Decrease $s_0$ so that $2s_0<\rho$. If
$|x_1|\le\rho/2$ and $s\le x_2\le2s$, then
\eqref{eq:lower-bound-flat-identity} gives
$d_M(x)=d_{A_{t,N}}(x)=x_2$. Moreover, $x_2<\rho<r$, so
\[
 w_s(d_{A_{t,N}}(x))
 =\exp\left\{-\frac{x_2^2}{2s^2}\right\}\ge e^{-2}.
\]
Restricting the preceding integral to this rectangle yields
\begin{align*}
 W_2^2(\nu_{A_{t,N},s^{-2}},\pi_M)
 &\ge \frac{e^{-2}}{\mathcal Z_{A_{t,N}}(s)}
 \int_{-\rho/2}^{\rho/2}\int_s^{2s}x_2^2\,dx_2\,dx_1\\
 &\ge \frac{\rho e^{-2}s^3}{\mathcal Z_{A_{t,N}}(s)}.
\end{align*}
Since \eqref{eq:lower-bound-normalizer-comparison} gives
$\mathcal Z_{A_{t,N}}(s)\le C(s+t^2)$, we obtain
\begin{equation}
 W_2^2(\nu_{A_{t,N},s^{-2}},\pi_M)
 \ge c\frac{s^3}{s+t^2},
 \qquad 0<s\le s_0.
\label{eq:lower-bound-normal-spread}
\end{equation}

\paragraph{Step 4: optimization over the penalty strength.}
Suppose first that $t^2\le s\le s_0$. Equations
\eqref{eq:lower-bound-tangential-deficit} and
\eqref{eq:lower-bound-normal-spread} imply
\[
 W_2(\nu_{A_{t,N},s^{-2}},\pi_M)
 \ge c\max\left\{\frac{t^2}{s},s\right\}
 \ge ct.
\]
If $0<s<t^2$, the tangential bound is bounded below by a fixed positive
constant, and therefore by $ct$ after decreasing $t_0$.

Finally, if $s\ge s_0$, then
$0\le\chi(d_{A_{t,N}}^2)\le\chi(4r^2)$ on $K$, so the density of
$\nu_{A_{t,N},s^{-2}}$ with respect to Lebesgue measure is bounded below
uniformly in $t$, $N$, and $s$ by
\[
 \frac{1}{\Leb_2(K)}
 \exp\left\{-\frac{\chi(4r^2)}{2s_0^2}\right\}.
\]
Since $\int_Kd_M(x)^2\,dx>0$, the coupling lower bound used above gives
$W_2(\nu_{A_{t,N},s^{-2}},\pi_M)\ge c_0>0$, which is again at least $ct$ for
small $t$. Thus the bound $ct$ holds for every $s>0$, equivalently for
every $\kappa>0$. The upper estimate
$\Delta\le Ct^2$ in \eqref{eq:lower-bound-basic-geometry} then gives
$ct\ge c\sqrt\Delta$, completing the proof.
\end{proof}

\subsection{Proof of Proposition~\ref{prop:torus-corollary}}

\begin{proof}
The map $F$ is a smooth embedding of the two-dimensional torus into
$\R^3$, so $D=3$, $d=2$, and $q=1$. Its image is compact, connected, and
without boundary. A unit normal is
\[
 n(\theta,\phi)
 =\bigl(\cos\theta\cos\phi,
        \cos\theta\sin\phi,
        \sin\theta\bigr),
\]
and the normal-coordinate map is
\[
 F(\theta,\phi)+t n(\theta,\phi)
 =\bigl((3+(1+t)\cos\theta)\cos\phi,
        (3+(1+t)\cos\theta)\sin\phi,
        (1+t)\sin\theta\bigr).
\]
This map is injective for $|t|<1$: the radial coordinate is then strictly
positive, and the toroidal angle $\phi$, the polar angle $\theta$, and $t$
are uniquely recovered. At $t=-1$, every value of $\theta$ with a fixed
$\phi$ maps to the corresponding point of the central circle. Hence
$\Reach(M)=1$, and Assumption~\ref{ass:true-manifold} holds with
$\tau_0=1$. The surface element is
\[
 \Vol_M(dy)=(3+\cos\theta)\,d\theta\,d\phi,
 \qquad V_M=12\pi^2.
\]
Thus uniform surface observations can be generated by taking $\phi$
uniformly on $[0,2\pi)$ and independently drawing $\theta$ with density
\[
 p_{\mathrm{obs}}(\theta)=\frac{3+\cos\theta}{6\pi},
 \qquad 0\le\theta<2\pi.
\]

Since $M\subseteq\overline B(0,4)$,
Assumption~\ref{ass:ambient-envelope} holds with $z_0=0$ and $R_0=4$.
Taking $r=1/4$ gives $2r<\tau_0$ and
$K=\overline B(0,19/4)$, from which the stated values of $L_K$ and $|K|$
follow. The potential $g(x)=x_3$ is smooth on $\R^3$, with
$\nabla g=e_3$ and $\nabla^2g=0$. This proves
$G=1$, $H=0$, and $\osc_K(g)=19/2$, verifying
Assumption~\ref{ass:potential-target}. Also,
\[
 \begin{aligned}
 Z_M
 &=\int_0^{2\pi}\int_0^{2\pi}
   e^{-\sin\theta}(3+\cos\theta)\,d\theta\,d\phi\\
 &=12\pi^2I_0(1),
 \end{aligned}
\]
which gives the displayed angular density.

The polynomial $w$ satisfies $w(0)=1$, $w(1)=0$,
$w'(0)=w'(1)=0$, and $0\le w\le1$ on $[0,1]$. Therefore the displayed
$\chi$ is nondecreasing and $C^2$, equals the identity on $[0,r^2]$, and
is constant on $[4r^2,\infty)$. Furthermore,
\[
 C_\chi=1,
 \qquad
 \chi(4r^2)
 =r^2+3r^2\int_0^1w(t)\,dt
 =\frac52r^2=\frac{5}{32},
\]
so \eqref{eq:cap-inner}--\eqref{eq:cap-lipschitz} hold.

Substitution into the constants preceding
Theorem~\ref{thm:rwmh-end-to-end} gives
\begin{equation}
 C_d=27,
 \qquad
 A_M=81e^{1/4},
 \qquad
 \kappa_0=162e^{1/4}\approx208.01.
\label{eq:torus-analysis-constants}
\end{equation}
Indeed, $8q/r^2=128$, whereas $2A_Mq=162e^{1/4}>128$.
The minorization constant \eqref{eq:rwmh-deterministic-minorization} becomes
\begin{equation}
 \underline\varepsilon_{\kappa,\sigma}
 =\frac{6859\pi}{48}(2\pi\sigma^2)^{-3/2}
 \exp\left\{
 -\frac{361}{8\sigma^2}-\frac{19}{2}-\frac{5\kappa}{64}
 \right\}.
\label{eq:torus-minorization}
\end{equation}

Condition \eqref{eq:torus-target-range} implies both
$\kappa_N\ge\kappa_0$ and $\epsilon_{\mathrm{tar}}^2\le r/4$.
Because $d=2$ and $\beta_d=2$, the reconstruction condition in
Corollary~\ref{cor:rwmh-expected-rate} is exactly
\eqref{eq:reconstruction-sample-size} with $\beta_d=2$. Here $C_H,C_P$,
and $N_0$ are the fixed-model constants supplied by
Proposition~\ref{prop:reconstruction-event}.
Substituting \eqref{eq:torus-minorization} into
\eqref{eq:rwmh-optimized-iterations} gives
\eqref{eq:torus-iteration-count}.

Therefore all hypotheses of
Corollary~\ref{cor:rwmh-expected-rate} are therefore satisfied. That is 
Assumptions~\ref{ass:true-manifold}--\ref{ass:potential-target} and
\eqref{eq:cap-inner}--\eqref{eq:cap-lipschitz} hold with
\begin{equation}
 \begin{gathered}
 D=3,\quad d=2,\quad q=1,\quad \tau_0=1,\quad z_0=0,
 \quad R_0=4,\quad r=\frac14,\\
 K=\overline B\left(0,\frac{19}{4}\right),\quad
 L_K=\frac{19}{2},\quad |K|=\frac{6859\pi}{48},\\
 G=1,\quad H=0,\quad \osc_K(g)=\frac{19}{2},\quad
 C_\chi=1,\quad \chi(4r^2)=\frac{5}{32}.
 \end{gathered}
\label{eq:torus-model-constants}
\end{equation}
Moreover,
\[
 Z_M=12\pi^2I_0(1),
 \qquad
 \pi_M(d\theta,d\phi)
 =\frac{e^{-\sin\theta}(3+\cos\theta)}
 {12\pi^2I_0(1)}\,d\theta\,d\phi,
\]
where $I_0$ is the modified Bessel function of the first kind.
This gives the claimed expected Wasserstein error and the scaling
$N=\widetilde O(\epsilon_{\mathrm{tar}}^{-2})$.
\end{proof}

\bibliographystyle{plainnat}
\bibliography{neurips_2026}

\end{document}